\documentclass[11pt,a4paper]{article}

\usepackage[a4paper,margin = 2.5cm]{geometry}
\usepackage{amsmath,amsthm,amssymb}
\usepackage{enumerate}
\usepackage{multicol}
\usepackage{graphicx}
\usepackage{tikz}
\usepackage{ifpdf}   
\AtBeginDocument{\gdef~{\nobreakspace{}}\catcode`\"=12}  
\usepackage[all,tips]{xy}  
\CompileMatrices  
\usepackage[cp1252]{inputenc}    
\usepackage{fontenc}             
\usepackage{microtype}

\newcommand{\lemref}[1]{Lemma~\ref{#1}}

\newcommand{\implica}{\mathbin{\rightarrow}}
\newcommand{\leftmodels}{%
  \mathrel{\text{\reflectbox{$\models$}}}%
}

\newcommand{\BL}{{\rm BL}} 
\newcommand{\BLpt}{{\rm BL$\forall$}}
\newcommand{\mBLpt}{{\rm mBL$\forall$}}
\newcommand{\ScBL}{{\rm S5(BL)}}
\newcommand{\ScpBL}{{\rm S5'(BL)}}

\DeclareMathOperator{\Fg}{Fg}
\DeclareMathOperator{\FMg}{MFg}
\DeclareMathOperator{\Con}{\mathit{Con}}
\DeclareMathOperator{\Conn}{\mathbf{Con}}

\theoremstyle{plain}

\newtheorem{thm}{Theorem}[section]

\newtheorem{lem}[thm]{Lemma}
\newtheorem{cor}[thm]{Corollary}

\theoremstyle{definition}
\newtheorem{dfn}[thm]{Definition}
\newtheorem{example}[thm]{Example}

\theoremstyle{remark}
\newtheorem{obs}[thm]{Remark}

\begin{document}

\title{Monadic BL-algebras: the equivalent algebraic semantics of Hájek's monadic fuzzy logic}

\author{Diego Castaño$^{a}$\footnote{Corresponding author. E-mail addresses:  diego.castano@uns.edu.ar (D. Castaño), crcima@criba.edu.ar (C. Cimadamore), usdiavar@criba.edu.ar (J. P. Díaz Varela), larueda@criba.edu.ar (L. Rueda)} \ , Cecilia Cimadamore$^{a}$, José Patricio Díaz Varela$^{a}$, Laura Rueda$^{a}$}

\maketitle

$^a$ Departamento de Matemática, Universidad Nacional del Sur, Bahía Blanca 8000, Argentina. \\ INMABB - CONICET, Bahía Blanca 8000, Argentina.

\begin{abstract}
In this article we introduce the variety of monadic BL-algebras as BL-algebras endowed with two monadic operators $\forall$ and $\exists$. After a study of the basic properties of this variety we show that this class is the equivalent algebraic semantics of the monadic fragment of Hájek's basic predicate logic. In addition, we start a systematic study of the main subvarieties of monadic BL-algebras, some of which constitute the algebraic semantics of well-known monadic logics: monadic Gödel logic and monadic \L ukasiewicz logic. In the last section we give a complete characterization of totally ordered monadic BL-algebras.
\end{abstract}


\section{Introduction}

In his book \cite{Hajek98libro} Hájek introduced BL-algebras as the algebraic semantics of his basic fuzzy logic, which is a common framework for \L ukasiewicz, Gödel and product logics. Afterwards in \cite{CiEsGoTo00} it was shown that Hájek's basic logic was the logic of continuous t-norms (see also \cite{CiTo05}). Subsequently BL-algebras were studied in great depth, see e.g. \cite{AgMo03,BuMo11handbook}. BL-algebras were also seen to be the subvariety of bounded integral commutative divisible residuated lattices generated by chains (see \cite{Hajek98libro}).

In \cite{Hajek98libro} Hájek also introduced the basic many-valued predicate logic and proved its strong completeness with respect to its (linear) general semantics, that is, the semantics based on Kripke frames where the
accessibility relation is total and the truth values lie on a BL-chain. A brief description of the monadic fragment of this calculus is also presented. Recall that the monadic fragment consists of the formulas with unary predicates and just one object variable. In addition, Hájek introduced an S5-like modal fuzzy logic and showed that it is equivalent to the monadic basic predicate logic. He also proposed a set of axioms and inference rules for the monadic logic and proved its strong completeness with respect to its (linear) general
semantics in \cite{Hajek10}.

Monadic algebras have been studied since Halmos introduced monadic
Boolean algebras in \cite{Halmos56}. Monadic versions of other algebraic
structures have been also greatly studied since then. The two most
important examples are monadic MV-algebras and monadic Heyting
algebras. The former were first studied by Rutledge in \cite{Rutledge59} and then by Di Nola, Grigolia, Cimadamore and Díaz Varela in \cite{DNGr04,CiDV14}. The latter were
introduced by Monteiro and Varsavsky in \cite{MoVa57} and deeply studied by Bezhanishvili in \cite{Bezhanishvili98}. Monadic \L ukasiewicz-Moisil algebras were
also studied by Abad in \cite{Abad88} and Heyting
algebras with one quantifier were the research topic of Rueda in \cite{Rueda01}.

In this article we will introduce the variety of monadic
BL-algebras\footnote{We should warn the reader that a different kind
of algebraic structures were introduced by Grigolia in \cite{Grigolia06} under
the same name. However, it may be seen that they are not the
equivalent algebraic semantics for Hájek's monadic calculus.}and
make a standard study of their basic properties, which includes the
characterization of their congruences and subdirectly irreducible
algebras. This will be the main topic of Section 2. In Section 3 we will give a complete characterization of the range of the monadic operators: $m$-relatively complete subalgebras. This characterization will be useful to produce the most important examples of monadic BL-algebras which are functional monadic BL-algebras. In the next section, Section 4, we will show that this variety is the equivalent algebraic semantics of Hájek's
monadic basic fuzzy logic in the sense of Blok and Pigozzi \cite{BlPi89} as well as simplify the original axioms proposed by Hájek. In Section 5 we will see that monadic BL-algebras contain as subvarieties the variety of monadic MV-algebras and monadic Gödel-algebras, the latter being monadic prelinear Heyting algebras that satisfy the equation $\forall(\exists x \vee y) \approx \exists x \vee \forall y$. We will also introduce the subvariety of monadic
product algebras and give a special characterization of its subdirectly irreducible members. In addition, in each of these three main subvarieties will give a complete characterization of their totally ordered members. Moreover, we devote Section 6 to study totally ordered monadic BL-algebras in depth. Specifically
we will show how to define all possible quantifiers on a given BL-chain. Finally, we conclude the paper describing  some of the problems about this variety that constitute our current work.

Throughout this article we assume that the reader is familiar with
propositional as well as first order basic logic and with structural
properties of BL-algebras.

\section{Monadic BL-algebras: definition and representation theorems}

We start this section with the definition of the variety
$\mathbb{MBL}$ of monadic BL-algebras. We develop the basic
arithmetical properties and prove that the image of the quantifier
is a subalgebra. We also introduce the notion of monadic filter and show that they correspond to congruences. As a
corollary, we derive a characterization for subdirectly irreducible
algebras and discuss a special BL-subdirect representation for them.
We refer the reader to \cite{Hajek98libro} for the definition and basic
properties of BL-algebras.

\newcounter{saveenum_mbl}

\begin{dfn}An algebra $\mathbf{A} = \langle A, \vee, \wedge, *,\implica,\exists, \forall, 0,1 \rangle$ of type
$(2,2,2,2,1,1,0,0)$ is called a \emph{monadic BL-algebra} (an
MBL-algebra for short) if $\langle A,\vee, \wedge, *,\implica,
0,1\rangle$ is a BL-algebra and the following identities are
satisfied:
\begin{enumerate}[(M1)]

 \item\label{M1} $\forall x\implica x\approx 1$.

 \item\label{M2} $\forall ( x\implica \forall y)\approx \exists x\implica \forall y$.

 \item \label{M3}$\forall (\forall x\implica y)\approx \forall x\implica \forall y$.

 \item\label{M4} $\forall ( \exists x\vee y)\approx \exists x\vee \forall y$.

 \item\label{M5} $\exists (x*x)\approx \exists x*\exists x$.

\setcounter{saveenum_mbl}{\value{enumi}}
\end{enumerate}
\end{dfn}

For brevity, if $\mathbf{A}$ is a BL-algebra and we enrich it with a monadic structure, we denote the resulting algebra by $\langle \mathbf{A}, \exists, \forall\rangle$. We denote by $\mathbb{MBL}$ the variety of MBL-algebras. The next
lemma collects some of the basic properties that hold true in any
MBL-algebra.

\begin{lem}\label{properties MBL} Let $\mathbf A\in \mathbb{MBL}$ and $a,b\in A$.
\begin{enumerate}[{\rm (M1)}]
\setcounter{enumi}{\value{saveenum_mbl}}
\begin{multicols}{2}

\item\label{M6} $\forall \exists a=\exists a$.

\item\label{M7} $a\implica \exists a=1$.

\item \label{M8} $\forall (\exists a\implica b)=\exists a\implica \forall b$.

 \item\label{M9} $\forall ( a\implica \exists b)= \exists a\implica \exists b$.

 \item\label{M10} $\forall 1=1$.

 \item\label{M11} $\exists \forall a=\forall a$.

 \item\label{M12} $\forall ( \forall a\vee b)= \forall a\vee \forall b$.

\item\label{M13} $\forall 0=0$, $\exists 1=1$, and $\exists 0 =0$.

\item\label{M14} $\exists \exists a = \exists a$ and $\forall \forall a = \forall
a$.
\item\label{M15} $ \forall(\exists a \implica  \exists b)= \exists a \implica \exists
b$.
\item\label{M16} $ \exists(\exists a \implica  b)\implica( \exists a \implica \exists
b) = 1$.
\item\label{M17} If $a\leq b$, then $\forall a \leq \forall b$ and $\exists a \leq  \exists
b$.
\item\label{M18} $\forall( \exists a \vee \exists b)= \exists a \vee \exists b$.
\item\label{M19} $\forall a = a$ if and only if $\exists a =a$.
\item\label{M20} $\exists( a \vee  b)= \exists a \vee \exists b$.
\item\label{M21} $\exists (\exists a * \exists b)= \exists a * \exists b$.
\item\label{M22} $\forall (a\implica b)\implica (\forall a\implica \forall b)=1$.
\item\label{M23} $\forall (a\implica b)\implica (\exists a\implica \exists b)=1$.
\item\label{M24} $(\forall a*\exists b)\implica \exists (a*b)=1$.
\item\label{M25} $(\forall a*\forall b)\implica \exists (a*b)=1$.
\item\label{M26} $\exists (a*\exists b)=\exists a*\exists b$.
\item\label{M27} $\exists (a*\forall b)=\exists a*\forall b$.
\item\label{M28} $\exists(a\implica \exists b) \implica (\forall a\implica \exists b)=1$.
\item\label{M29} $\exists (\exists a\implica \exists b)=\exists a\implica \exists b$.
\item\label{M30} $\exists (\forall a\implica \forall b )=\forall a\implica \forall b$.
\item\label{M31} $\exists (\exists a\wedge \exists b)=\exists a\wedge \exists b$.
\item\label{M32}  $\exists ( a\wedge \exists b)=\exists a\wedge \exists b$.
\item\label{M33} $\forall (\forall a\implica \forall b )= \forall a\implica \forall b $.
\item\label{M34} $\exists (\forall a *\forall b )= \forall a* \forall b$.
\item\label{M35} $\forall (\forall a *\forall b )= \forall a* \forall b$.
\item\label{M36} $\forall (\forall a \wedge\forall b )= \forall a\wedge \forall b$.
\item\label{M37} $\forall (a \wedge b )= \forall a\wedge \forall b$.
\end{multicols}
\setcounter{saveenum_mbl}{\value{enumi}}
\end{enumerate}
\end{lem}
\begin{proof}
\begin{enumerate}[{\rm (M1)}]
\setcounter{enumi}{5}
\item 
From $\exists a=\exists a\vee \exists a$ and (M\ref{M4}), we have
that $\forall \exists a=\forall (\exists a\vee \exists a)= \exists
a\vee \forall \exists a$. But from (M\ref{M1}) we know that $
\forall \exists a\leq\exists a $. So, $\forall \exists a=\exists a$.

\item  

From (M\ref{M6}), (M\ref{M2}) and (M\ref{M1}), we can write
$1=\exists a\implica \forall \exists a= \forall (a\implica \forall
\exists a)\leq a\implica \forall \exists a=a\implica \exists a$.
Thus, $a\implica \exists a= 1$.

\item 
From (M\ref{M6}) and (M\ref{M3}) we have that $\forall ( \exists
a\implica b)=  \forall (\forall \exists a\implica b)= \forall
\exists a\implica \forall b=\exists a\implica \forall b$.

\item 
From (M\ref{M6}) and (M\ref{M2}) we have that $\forall (a\implica
\exists b)= \forall (a\implica \forall\exists b)=\exists a\implica
\forall\exists b=\exists a\implica \exists b$.

\item 

From (M\ref{M1}) and (M\ref{M3}), we have that $\forall 1= \forall
(\forall 1\implica 1)= \forall 1\implica \forall 1=1$.

\item 
We know that $\forall a\implica \exists \forall a=1$ by (M\ref{M7}).
Furthermore, $\exists \forall a\implica \forall a= \forall (\forall
a\implica \forall a)= \forall 1=1$ by (M\ref{M2}) and (M\ref{M10}).
Then, $\exists \forall a= \forall a$.

\item 
Using (M\ref{M4}) and (M\ref{M11}), we have that $\forall ( \forall
a\vee b)=\forall ( \exists \forall a\vee b)= \exists \forall a\vee
\forall b= \forall a\vee \forall b$.

\item Clearly $\forall 0=0 $ and $\exists 1=1$ by (M\ref{M1}) and (M\ref{M7}) respectively.

Since
$\forall 0=0$, we have that $\exists 0 \implica 0=\exists 0 \implica
\forall 0 = \forall ( 0 \implica \forall 0) = \forall ( 0 \implica
0)=  \forall 1=1$ by (M\ref{M2}) and (M\ref{M10}). So, $\exists
0=0$.

\item By (M\ref{M7}) we have that $\exists a\implica \exists\exists a=1$. On the other hand, from (M\ref{M9}) and (M\ref{M10}) we have that $\exists \exists a \implica \exists a = \forall(\exists a\implica  \exists a) = \forall 1 =1$. Thus, $\exists\exists a=\exists a$.

By (M\ref{M3}) and (M\ref{M10}), $\forall a \implica \forall \forall
a= \forall(  \forall a \implica \forall a)= \forall 1 = 1$. Since
$\forall \forall a\implica \forall a=1$, then $\forall \forall a=
\forall a$.

\item From (M\ref{M9}) and (M\ref{M14}), we have that $\forall( \exists a \implica \exists b)=  \exists\exists a \implica \exists b = \exists a \implica \exists b$.

 \item Using (M\ref{M15}), (M\ref{M2}), and $\exists a \implica  b \leq \exists a \implica \exists b$, we obtain $\exists(\exists a \implica  b)  \implica ( \exists a \implica \exists b) = \exists(\exists a \implica  b)\implica \forall ( \exists a \implica
\exists b) = \forall ( (\exists a \implica  b) \implica \forall( \exists a \implica \exists b)) = \forall ( (\exists a \implica  b) \implica ( \exists a \implica \exists b)) = \forall 1= 1$.

\item If $a \leq b$, then $\forall a \leq a \leq b \leq \exists b$. Thus, by (M\ref{M3}), $\forall a
\implica \forall b = \forall(\forall a \implica b ) = \forall 1 = 1$, and, by (M\ref{M9}), $\exists a \implica \exists b = \forall
(a\implica \exists b) = \forall 1 = 1$.

\item Using (M\ref{M4}) and (M\ref{M6}), $\forall (\exists a \vee \exists b)=  \exists a \vee \forall \exists b = \exists a \vee \exists b$.

\item If $a = \forall a$, then $\exists a = \exists \forall a=  \forall a=
a$ by (M\ref{M11}). The converse implication follows analogously using (M\ref{M6}).

\item Clearly $a\vee b \leq \exists a \vee \exists
b$. Thus $\exists (a\vee b) \leq \exists(\exists a \vee \exists b)$,
by (M\ref{M17}). But $\exists a \vee \exists b= \forall(\exists a
\vee \exists b) $ by (M\ref{M18}). Then, taking (M\ref{M19}) into
account, $\exists (\exists a \vee \exists b) =  \exists a \vee
\exists b$. Therefore, $\exists (a\vee b) \leq \exists a \vee
\exists b$.

On the other hand, from $\exists a \leq \exists (a\vee b)$ and
$\exists b \leq \exists (a\vee b)$, it is clear that $\exists a \vee
\exists b \leq \exists (a\vee b)$.

\item Taking into account that $\exists a * \exists b \geq \forall(\exists a * \exists
b)$, (M\ref{M2}) and (M\ref{M8}), we have that $\exists(\exists a * \exists b)
\implica (\exists a * \exists b) \geq \exists(\exists a * \exists b)
\implica \forall(\exists a * \exists b)= \forall((\exists a *
\exists b) \implica \forall(\exists a * \exists b)) = \forall (\exists a \implica( \exists b \implica
\forall(\exists a * \exists b))) = \forall (\exists a \implica
\forall( \exists b \implica (\exists a * \exists b))) =
\forall\forall (\exists a \implica ( \exists b \implica (\exists a
* \exists b))) = \forall \forall 1 = 1$.

\item From $\forall a\leq a$, we have that $a\implica b\leq \forall a\implica b$. Then, using (M\ref{M17}) and (M\ref{M3}), $\forall (a\implica b)\leq \forall (\forall a\implica b) = \forall a\implica \forall b$.

\item Since $b\leq \exists b$, we have that $a\implica b\leq a\implica \exists b$. Thus, $\forall (a\implica b)\leq \forall (a\implica \exists b)$. Then, from (M\ref{M9}), we have that $\forall (a\implica b)\implica (\exists a\implica \exists b)= \forall (a\implica b)\implica \forall( a\implica \exists b)=1$.

\item From (M\ref{M23}), (M\ref{M22}) and (M\ref{M10}), we have that $(\forall a*\exists b)\implica \exists (a*b) = \forall a\implica (\exists b\implica \exists (a*b)) \geq \forall a\implica \forall(b\implica (a*b)) \geq \forall (a\implica (b\implica (a*b))) = \forall 1 = 1$.

\item Using (M\ref{M7}), $(\forall a*\forall b)\implica \exists (a*b)\geq (\forall a*\forall b)\implica  (a*b)=1$.

\item Since $a\leq \exists a$, then $a*\exists b\leq \exists a*\exists b$. Consequently, $\exists (a*\exists b)\leq \exists (\exists a*\exists b)=\exists a*\exists b$ from (M\ref{M21}). On the other hand, from (M\ref{M9}) and (M\ref{M8}), we have that $(\exists a* \exists b)\implica \exists (a*\exists b)= \exists b\implica (\exists a\implica \exists (a*\exists b))=\exists b\implica \forall (a\implica \exists (a*\exists b))= \forall (\exists b\implica (a\implica \exists (a*\exists b)))=\forall ( (a*\exists b)\implica \exists (a*\exists b))= \forall 1=1$.

\item Using (M\ref{M11}) and (M\ref{M26}), we have that $\exists(a * \forall b) = \exists (a * \exists \forall b) = \exists a * \exists \forall b = \exists a *  \forall b$.

\item From (M\ref{M2}), $\exists(a\implica \exists b) \implica (\forall a\implica \exists b)=\exists(a\implica \exists b) \implica \forall(\forall a\implica \exists b)=\forall( (a\implica \exists b) \implica \forall(\forall a\implica \exists b))=\forall( (a\implica \exists b) \implica (\forall a\implica \exists b)) = \forall 1 = 1$.

\item Clearly $\exists a\implica \exists b\leq \exists (\exists a\implica \exists b)$. Using (M\ref{M28}) and (M\ref{M6}), $\exists(\exists a \implica \exists b) \leq \forall \exists a \implica \exists b = \exists a \implica \exists b$.

\item Using (M\ref{M11}) and (M\ref{M29}), $\exists(\forall a \implica \forall b) = \exists(\exists \forall a \implica \exists \forall b) = \exists \forall a \implica \exists \forall b = \forall a \implica \forall b$.

\item From (M\ref{M7}) we have that $\exists ( \exists a \wedge \exists b)\geq\exists a\wedge \exists b$. Since $\exists a\wedge \exists b\leq \exists a$ and $\exists a\wedge \exists b\leq \exists b$, then $\exists (\exists a\wedge \exists b)\leq \exists\exists a=\exists a$ and $\exists(\exists a\wedge \exists b)\leq \exists \exists b=\exists b$. Thus, $\exists (\exists a\wedge \exists b)\leq \exists a\wedge \exists b$.

\item We know that $a\wedge \exists b\leq \exists a \wedge \exists b$, then $\exists (a\wedge \exists b)\leq \exists (\exists a \wedge \exists b)=\exists a \wedge \exists b$ by (M\ref{M31}). On the other hand, $(\exists a \wedge \exists b)\implica \exists (a\wedge \exists b)= (\exists a*(\exists a\implica \exists b))\implica \exists (a* (a\implica \exists b))= (\exists a*\forall( a\implica \exists b))\implica \exists (a* (a\implica \exists b))= 1$ by (M\ref{M9}) and (M\ref{M24}).

\item From (M\ref{M3}), $\forall (\forall a\implica \forall b )= \forall a\implica \forall \forall b = \forall a\implica \forall b$.

\item Using (M\ref{M11}) and (M\ref{M21}), we have that $\exists (\forall a*\forall b)=\exists (\exists \forall a*\exists\forall b)=\exists \forall a*\exists\forall b= \forall a*\forall b  $.

\item Using (M\ref{M34}) and (M\ref{M6}), we have that $\forall a*\forall b=\exists(\forall a*\forall b)=\forall (\exists(\forall a*\forall b))= \forall(\forall a*\forall b)$.

\item  Using (M\ref{M33}) and (M\ref{M35}), we have that $\forall (\forall a \wedge\forall b )=\forall (\forall a*(\forall a\implica \forall b)) =\forall (\forall a*\forall(\forall a\implica \forall b))=\forall a*\forall(\forall a\implica \forall b)=\forall a*(\forall a\implica \forall b)=\forall a\wedge \forall b$.

\item From (M\ref{M36}), $(\forall a\wedge \forall b)\implica \forall (a \wedge b )= \forall (\forall a\wedge \forall b)\implica \forall (a \wedge b )= \forall ((\forall a\wedge \forall b)\implica (a \wedge b ))=\forall 1=1$. Then,  $\forall a\wedge \forall b\leq  \forall (a \wedge b )$. Since $a\wedge b\leq a$ and $a\wedge b\leq b$, we have $ \forall (a \wedge b )\leq \forall a\wedge \forall b$. \qedhere
\end{enumerate}
\end{proof}

\begin{obs} \rm
Observe that from (M\ref{M6}) and (M\ref{M11}), $x = \forall y$ for some $y$ if and only if $x = \exists z$ for some $z$. Thus $\exists A = \{\exists a: a \in A\} = \{\forall a: a \in A\} = \forall A$. Using this fact, from (M\ref{M2}), (M\ref{M3}) and (M\ref{M4}), we deduce that:
\begin{itemize}
\item $\forall (a \implica c) = \exists a \implica c$,
\item $\forall (c \implica a) = c \implica \forall a$,
\item $\forall (c \vee a)= c \vee \forall a$,
\end{itemize} 
for any $a \in A$ and $c \in \forall A$. Compare with the axioms of the logic $S5(\mathcal C)$ in Section 4.
\end{obs}

\begin{obs}
Note that the identity $\forall (x * x) \approx \forall x * \forall x$ holds in every monadic MV-algebra and trivially in any monadic Gödel algebra (cf. Section 5). However, this equation is not valid in any monadic BL-algebra. For example, let $\mathbf{L}_2$ and $\mathbf{L}_3$ be the 2-element and 3-element MV-chains, respectively, and consider the ordinal sum $\mathbf{A} = \mathbf{L}_3 \oplus \mathbf{L}_2$. If we define the quantifiers on $A$ in the following way:
\begin{center}
\begin{tabular}{|c|c|c|c|c|}
  \hline
  $a$         & $0_1$ & $\frac{1}{2}$ & $0_2$         & $1$ \\ \hline
  $\forall a$ & $0_1$ & $\frac{1}{2}$ & $\frac{1}{2}$ & $1$ \\
  $\exists a$ & $0_1$ & $\frac{1}{2}$ & $1$           & $1$ \\
  \hline
\end{tabular}
\end{center}
we get a monadic BL-chain $\langle \mathbf{A}, \exists, \forall\rangle$. In this algebra $\forall 0_2*\forall 0_2 = \frac{1}{2} * \frac{1}{2} = 0_1$, but $\forall (0_2*0_2)= \forall 0_2=\frac{1}{2}$.
\end{obs}

\begin{lem} \label{LEMA: para todo A es subalgebra}
If $\mathbf A \in \mathbb{MBL}$, then $\exists A=\forall A$ and
$\exists \mathbf{A}$ is a BL-subalgebra of $\mathbf{A}$.
\end{lem}

\begin{proof}
We already showed that $\forall A = \exists A$. From (M\ref{M7}), (M\ref{M13}), (M\ref{M20}), (M\ref{M31}),
(M\ref{M21}) and (M\ref{M29}), we obtain that $\exists \mathbf{A}$ is a BL-subalgebra of $\mathbf{A}$.
\end{proof}

Given a monadic BL-algebra $\mathbf{A}$, a subset $F \subseteq A$ is a {\em monadic filter} of $\mathbf{A}$ if $F$ is a filter and $\forall a \in F$ for each $a \in F$. In the following, we characterize the congruences of each MBL-algebra $\mathbf{A}$ by means of monadic filters in the standard way. More precisely, we establish an order isomorphism from the lattice $\Conn_\mathbb{MBL}(\mathbf{A})$ of congruences of $\mathbf{A}$ onto the lattice $\mathbf{F}_m(\mathbf{A})$ of monadic filters of $\mathbf{A}$, both ordered by inclusion. Moreover, we prove that the lattice $\mathbf{F}_m(\mathbf{A}) $ is isomorphic to the lattice $\mathbf{F}(\exists \mathbf{A})$ of filters of the BL-algebra $\exists \mathbf{A}$.

If $\mathbf{A} \in \mathbb{MBL}$ and $X \subseteq A$, $X \neq \emptyset$, the monadic filter generated by $X$ is the set
\begin{equation*}
\begin{split}
\FMg(X) & =\left\{a\in A: \forall x_1\implica(\forall x_2\implica(\cdots (\forall x_n\implica a)\cdots ))=1, \text{ where } x_1,x_2,\dots,x_n\in X \right\} \\
& = \left\{a\in A: \forall x_1 * \forall x_2 * \cdots *
\forall x_n\leq a,\text{ where  } x_1,x_2,\dots,x_n\in  X \right\}.
\end{split}
\end{equation*}
In particular, if $X = \{b\}$ then
$\FMg(\{b\}) = \FMg(b) = \left\{a \in A: (\forall b)^n \leq a \text{ for some } n \in \mathbb{N}\right\}$. Let us observe also
that $\FMg(X)=\Fg(\forall X)$, the filter generated by $\forall X$.

\begin{thm} Let $\mathbf{A}\in \mathbb{MBL}$. The correspondence  $\mathbf{Con}_\mathbb{MBL}(\mathbf{A})\to \mathbf{F}_m(\mathbf{A})$ defined by $\theta \mapsto 1/\theta$ is an order isomorphism whose inverse is given by $F \mapsto \theta_F=\left\{ (a,b)\in A^2: (a\implica b) *  (b\implica a) \in F\right\}$.
\end{thm}

\begin{proof}
Let $\theta \in \Con_{\mathbb{MBL}}(\mathbf{A})$. Let us consider the
filter $1/\theta$ and let $a \in 1/\theta$, that is, $a \equiv
1 \pmod{\theta}$. Since $\theta $ is a congruence, we have that
$\forall a \equiv \forall 1 \pmod{\theta}$ and from here we
clearly obtain that $\forall a \in 1/\theta$. Consequently,
$1/\theta$ is a monadic filter. Let $F \in
F_{m}(\mathbf{A})$. Let us prove that $\theta_F \in
\Con_{\mathbb{MBL}}(\mathbf{A})$. Indeed, let $a,b\in A$ such that
$(a \implica b) * (b \implica a) \in F$. Then, $a \implica b \in F$
and $b \implica a \in F$. So $\forall(a \implica b) \in F$ and
$\forall(b \implica a) \in F$. Since $F$ is increasing and from
(M\ref{M22}), we obtain that $\forall a \implica \forall b \in F$ and
$\forall b \implica \forall a \in F$. Thus, $(\forall a \implica \forall
b)*(\forall b \implica \forall a) \in F$. Furthermore, from
(M\ref{M23}) and considering again that $\forall(a \implica b) \in
F$, $\forall(b \implica a) \in F$ and $F$ is increasing, we have
that $(\exists a \implica \exists b) * (\exists b \implica \exists a) \in
F$. Thus, $\theta_F \in \Con_{\mathbb{MBL}}(\mathbf{A})$.
Now, it is straigthforward to see that the correspondence $\theta \mapsto
1/\theta$ is an order isomorphism whose inverse is given by $F \mapsto
\theta_F$.
\end{proof}

The following theorem is also routine.

\begin{thm}
Let $\mathbf{A} \in \mathbb{MBL}$. The correspondence $\mathbf{F}_{m}(\mathbf{A}) \to \mathbf{F}(\exists \mathbf{A})$ defined by $F \mapsto F \cap \exists A$ is an order isomorphism whose inverse is given by $M \mapsto \FMg(M)$.
\end{thm}

\begin{cor} \label{cor:iso_filtrosmonadicos_filtrosbl}
If $\mathbf{A} \in \mathbb{MBL}$ then \[ \Conn_{\mathbb{MBL}}(\mathbf{A})\cong \mathbf{F}_{m}(\mathbf{A}) \cong \mathbf{F}(\exists \mathbf{A}) \cong  \Conn_{\mathbb{BL}}(\exists \mathbf{A}). \]
\end{cor}

As an immediate consequence, we have the following results.

\begin{cor}
Let $\mathbf{A} \in \mathbb{MBL}$.
\begin{enumerate}[$(1)$]
\item $\mathbf{A}$ is subdirectly irreducible (simple) if and only if $\exists \mathbf{A}$ is a subdirectly irreducible (simple) BL-algebra.
\item If $\mathbf{A}$ is subdirectly irreducible, then $\exists \mathbf{A}$ is totally ordered.
\item $\mathbf{A}$ is a subdirect product of a family of MBL-algebras $\{\mathbf A_i: i \in I\}$ such that $\exists \mathbf A_i$ is totally ordered.
\end{enumerate}
\end{cor}

To close this section of basic properties we will show an extension to monadic BL-algebras of a representation theorem for monadic MV-algebras showed by Rutledge in \cite{Rutledge59}.

\begin{lem}
Let $\mathbf{A}$ be an MBL-algebra and $F$ be a filter in
$\mathbf{A}$. For any $x,y \in A$, $$F = \Fg(F \cup \{x \implica y\})
\cap \Fg(F \cup \{y \implica x\}).$$
\end{lem}

\begin{proof}
The forward inclusion is straightforward. Now assume $z$ is an
element of both $\Fg(F \cup \{x \implica y\})$ and $\Fg(F \cup \{y \to
x\})$. Then, there are $f_1,f_2 \in F$, $n_1,n_2 \in \mathbb{N}_0$
such that $f_1 * (x \implica y)^{n_1} \leq z$ and $f_2 * (y \implica x)^{n_2}
\leq z$. If we let $f = f_1*f_2$ and $n = \max\{n_1,n_2\}$, it
follows that $f * (x \implica y)^n \leq z$ and $f * (y \implica x)^n \leq z$.
Using the residuation condition, $(x \implica y)^n \leq f \implica z$ and  $(y
\implica x)^n \leq f \implica z$. Thus we get $(x \implica y)^n \vee (y \implica x)^n
\leq f \implica z$. But $(x \implica y)^n \vee (y \implica x)^n = ((x \implica y) \vee
(y \implica x))^n = 1$, so $f \leq z$ and $z \in F$.
\end{proof}

Recall that a filter $F$ in a BL-algebra $\mathbf{A}$ is said to be {\em prime} if for every $a,b \in A$, either $a \implica b \in F$ or $b \implica a \in F$. Observe that $F$ is prime if and only if $\mathbf{A}/F$ is totally ordered.

\begin{lem}
Let $\mathbf{A}$ be an MBL-algebra such that $\exists \mathbf{A}$ is
totally ordered. Given $a \in A$, $a \ne 1$, there exists a prime filter $P$ in $\mathbf{A}$ such that $a \vee \forall r
\not\in P$ for any $r \ne 1$.
\end{lem}

\begin{proof}
Consider $C = \{a \vee \forall r: r \ne 1\}$. Note that $1 \not\in
C$, since $a \vee \forall r = 1$ implies that $1 = \forall(a \vee
\forall r) = \forall a \vee \forall r$ and this, in turn, would
imply that $a = 1$ or $r = 1$.

Let $\mathcal{F}$ be the family of filters $F$ in
$\mathbf{A}$ such that $F \cap C = \emptyset$. The above paragraph
shows that $\{1\} \in \mathcal{F}$, so that $\mathcal{F}$ is
nonempty. In addition, it is straightfoward to verify that any chain
in $\mathcal{F}$ has an upper bound in $\mathcal{F}$. Hence, by
Zorn's Lemma, there exists a maximal filter $P$ in $\mathcal{F}$.

We claim that $P$ is prime. Indeed, let $x,y \in A$ and note that  $$P = \Fg(P \cup \{x \implica y\}) \cap \Fg(P
\cup \{y \implica x\}).$$ If we assume that neither $\Fg(P \cup \{x \to
y\})$ nor $\Fg(P \cup \{y \implica x\})$ belongs to $\mathcal{F}$, then
there are $r_1, r_2 \ne 1$ such that $a \vee \forall r_1 \in \Fg(P
\cup \{x \implica y\})$ and $a \vee \forall r_2 \in \Fg(P \cup \{y \to
x\})$. Since $\forall r_1$ and $\forall r_2$ are comparable, it
follows that one of them belongs to both filters. Hence one of them
belongs to $P$, a contradiction. This shows that either $ \Fg(P \cup
\{x \implica y\}) \in \mathcal{F}$ or $\Fg(P \cup \{y \implica x\}) \in
\mathcal{F}$.

Assume the first option is true. By the maximality of $P$, $P =
\Fg(P \cup \{x \implica y\})$, so $x \implica y \in P$. Analogously, if $P = \Fg(P \cup \{y \implica x\})$, $y \implica x
\in P$.
\end{proof}

\begin{thm}
Given an MBL-algebra $\mathbf{A}$, there exists a subdirect
representation of the underlying BL-algebra $\mathbf{A} \leq
\prod_{i \in I} \mathbf{A}_i$, where each $\mathbf{A}_i$ is a
totally ordered BL-algebra and $\exists \mathbf{A}$ is embedded in
$\mathbf{A}_i$ via the corresponding projection map.
\end{thm}

\begin{proof}
For each $a \in A$, $a \ne 1$, let $P_a$ be one of the prime filters provided by the previous lemma. Clearly $\bigcap_{a
\ne 1} P_a = \{1\}$ and we obtain a natural embedding $\mathbf{A}
\to \prod_{a \ne 1} \mathbf{A}/P_a$. To close the proof we need only
show that the natural map $\mathbf{A} \to \mathbf{A}/P_a$ is
injective on $\exists A$. Indeed, suppose there were $r_1,r_2 \in A$
such that $\forall r_1 < \forall r_2$ and $\forall r_1/P_a = \forall
r_2/P_a$. We have that $\forall r_2 \implica \forall r_1 =
\forall(\forall r_2 \implica \forall r_1)$ and $\forall r_2 \implica \forall
r_1  \ne 1$. Hence, we know that $a \vee (\forall r_2 \implica \forall
r_1) \not\in P_a$, which is a contradiction.
\end{proof}

\section{Building MBL-algebras from $m$-relatively complete subalgebras}

In this section we characterize those subalgebras of a given BL-algebra that may be the range of the quantifiers $\forall$ and $\exists$. As a consequence of this characterization, we build the most important examples of monadic BL-algebras, which we call {\em functional monadic BL-algebras}. This construction will allow us in the next section to prove the main result of this article: monadic BL-algebras are the equivalent algebraic semantics of Hájek's monadic fuzzy logic.

Given a BL-algebra $\mathbf{A}$, we say that a subalgebra
$\mathbf{C} \leq \mathbf{A}$ is {\em $m$-relatively complete} if the
following conditions hold:
\begin{enumerate}[(s1)]
\item For every $a \in A$, the subset $\{c \in C: c \leq a\}$ has a greatest element and $\{c \in C: c \geq a\}$ has a least element.

\item For every $a \in A$ and $c_1,c_2 \in C$ such that $c_1 \leq c_2 \vee a$, there exists $c_3 \in C$ such that $c_1 \leq c_2 \vee c_3$ and $c_3 \leq a$.

\item For every $a \in A$ and $c_1 \in C$ such that $a * a \leq c_1$, there exists $c_2 \in C$ such that $a \leq c_2$ and $c_2 * c_2 \leq c_1$.
\end{enumerate}

Condition (s2) may be replaced by either of the following:
\begin{enumerate}
\item[(s2$'$)] For every $a \in A$ and $c_1,c_2 \in C$ such that $c_1 = c_2 \vee a$, there exists $c_3 \in C$ such that $c_1 = c_2 \vee c_3$ and $c_3 \leq a$.
\item[(s2$''$)] If $1 = c_1 \vee a$ for some $c_1 \in C$, $a \in A$, there exists $c_2 \in C$ such that $1 = c_1 \vee c_2$ and $c_2 \leq a$.
\end{enumerate}

Indeed, (s2$'$) is an easy consequence of (s2) and (s2$''$) follows
from (s2$'$) by setting $c_1 = 1$. Now assume (s2$''$). If $c_1 \leq
c_2 \vee a$, it follows that $1 = (c_1 \implica c_2) \vee (c_1 \implica a)$.
Hence, there exists $c_3 \in B$ such that $1 = (c_1 \implica c_2) \vee
c_3$ and $c_3 \leq c_1 \implica a$. Thus $c_1 * c_3 \leq a$ and $c_1 \to
(c_2 \vee (c_1
* c_3)) = (c_1 \implica c_2) \vee (c_1 \implica (c_1 * c_3)) \geq (c_1 \to
c_2) \vee c_3 = 1$.

Furthermore observe that if $\mathbf{C}$ is totally ordered,
condition (s2$''$) may be replaced by the following simpler
equivalent form:
\begin{enumerate}
\item[(s2$_\ell''$)] If $1 = c \vee a$ for some $c \in C$, $a \in A$, then $c = 1$ or $a = 1$.
\end{enumerate}

\begin{thm} \label{TEO: caracterizacion de la imagen del cuantificador}
Given a BL-algebra $\mathbf{A}$ and an $m$-relatively complete
subalgebra $\mathbf{C} \leq \mathbf{A}$, if we define on $A$ the
operations $$\exists
a := \min \{c \in C: c \geq a\}, \qquad \forall a := \max \{c \in C: c \leq a\},$$ then $\langle \mathbf{A}, \exists, \forall\rangle$ is a monadic BL-algebra such that $\forall
A = \exists A = C$. Conversely, if $\mathbf{A}$ is a monadic
BL-algebra, then $\forall \mathbf{A} = \exists \mathbf{A}$ is an $m$-relatively
complete subalgebra of $\mathbf{A}$.
\end{thm}

\begin{proof}
Clearly condition (s1) from the definition of $m$-relatively
complete subalgebra guarantees the existence of $\forall a$ and
$\exists a$ for every $a \in A$. It remains to show that $\langle
\mathbf{A}, \exists, \forall\rangle$ satisfies axioms (M1)--(M5). Let
$a,b \in A$.
\begin{enumerate}[(M1)]
\item From the definition of $\forall a$, it is clear that $\forall
a \leq a$. Thus $\forall a \implica a = 1$.

\item Since, by definition, $a \leq \exists a$, it follows that
$\exists a \implica \forall b \leq a \implica \forall b$. Then $\exists a \to
\forall b \in \{c \in C: c \leq a \implica \forall b\}$. Let us see that
$\exists a \implica \forall b = \max \{c \in C: c \leq a \implica \forall
b\}$. Indeed, if $c \in C$ and $c \leq a \implica \forall b$, then $a
\leq c \implica \forall b$. Then, by definition of $\exists a$, $\exists
a \leq c \implica \forall b$. Thus $c \leq \exists a \implica \forall b$. This
shows that $\forall (a \implica \forall b) = \exists a \implica \forall b$.

\item From $\forall b \leq b$, we get $\forall a \implica \forall b \leq
\forall a \implica b$. In addition, if $c \in C$ and $c \leq \forall a
\implica b$, then $c * \forall a \leq b$. Thus $c * \forall a \leq
\forall b$ and $c \leq \forall a \implica \forall b$. Hence we have shown
that $\forall (\forall a \implica b) = \forall a \implica \forall b$.

\item Since $\forall b \leq b$, $\exists a \vee \forall b \leq
\exists a \vee b$. Now, if $c \in C$ and $c \leq \exists a \vee b$,
by condition (s2) in the definition of $m$-relatively complete
subalgebra, there must be $c' \in C$ such that $c \leq \exists a
\vee c'$ and $c' \leq b$. Then $c' \leq \forall b$ and $c \leq
\exists a \vee \forall b$. Thus, we have shown that $\forall(\exists
a \vee b) = \exists a \vee \forall b$.

\item We know that $a * a \leq \exists a * \exists a$. In addition,
if $c \in C$ and $a * a \leq c$, by condition (s3), there is $c'
\in C$ such that $c' * c' \leq c$ and $a \leq c'$. Then $\exists a
\leq c'$ and $\exists a * \exists a \leq c' * c' \leq c$. We have
thus proved that $\exists(a*a) = \exists a * \exists a$.
\end{enumerate}

Conversely, let $\langle \mathbf{A}, \exists, \forall\rangle$ be a
monadic BL-algebra. From Lemma \ref{LEMA: para todo A es
subalgebra}, we already know that $\forall\mathbf{A}$ is a BL-subalgebra of
$\mathbf{A}$. Let us now show that conditions (s1)--(s3) hold.

\begin{enumerate}[(s1)]
\item Using the properties from Lemma \ref{properties MBL}, we have
that if $c \leq a$ for some $c \in \forall A$, then $c = \forall c
\leq \forall a \leq a$. Thus $\forall a = \max\{c \in \forall A: c
\leq a\}$. Analogously $\exists a = \min\{c \in \forall A: c \geq
a\}$.

\item Assume $c_1 \leq c_2 \vee a$ for some $c_1, c_2 \in \forall
A$, $a \in A$. Then, using the properties in Lemma \ref{properties
MBL} and the axioms for monadic BL-algebras, we get that $c_1 \leq
c_2 \vee \forall a$ and $\forall a \leq a$.

\item Similarly to the previous paragraph, if $a * a \leq c$ for some $c \in \forall A$ and $a \in
A$, then $\exists a * \exists a \leq c$ and $a \leq \exists a$,
$\exists a \in \forall A$.
\end{enumerate}
This completes the proof that $\forall\mathbf{A}$ is an $m$-relatively
complete subalgebra of $\mathbf{A}$.
\end{proof}

The following is the most important example of monadic BL-algebras
built according to the previous theorem.

\begin{example} \label{EJ: functional MBL}
Consider a BL-chain $\mathbf{A}$ and a nonempty set $X$. We restrict our attention to those elements $f \in A^X$ such that $\inf\{f(x): x \in X\}$ and $\sup\{f(x):x \in X\}$ both exist in $\mathbf{A}$. We denote by $S$ the subset of $A^X$ of these \lq\lq safe\rq\rq \ elements. For every $f \in S$, we define: $(\forall_\wedge f)(x) = \inf\{f(y): y \in X\}$ and $(\exists_\vee f)(x) = \sup\{f(y): y \in X\}$, $x \in X$. Note that $\forall_\wedge f$ and $\exists_\vee f$ are constant maps.

Let $\mathbf{B}$ be a BL-subalgebra of $\mathbf{A}^X$ contained in $S$ such that for every $f \in B$, $\forall_\wedge f, \exists_\vee f \in B$. We claim that $\mathbf{B}$ has a natural structure of monadic BL-algebra.

Let $C$ be the subset of constant maps of $A^X$. We claim that $\mathbf{B} \cap \mathbf{C}$ is an $m$-relatively complete subalgebra of $\mathbf{B}$. Indeed, since $\mathbf{B}$ and $\mathbf{C}$ are subalgebras of $\mathbf{A}^X$, it is clear that $\mathbf{B} \cap \mathbf{C}$ is a subalgebra of $\mathbf{B}$.

If $f \in B$, then $\forall_\wedge f \in B$, so $\max \{c \in B \cap C: c \leq f\} = \forall_\wedge f \in B$. Analogously, $\min\{c \in B \cap C: c \geq f\} = \exists_\vee f \in B$. This shows that condition (s1) holds.

Since $\mathbf{B} \cap \mathbf{C}$ is totally ordered, we may check condition
(s2$_\ell''$) instead of (s2). Assume $1 = c \vee f$ for some $f \in
B$ and $c \in B \cap C$. Put $c(x) = c_0 \in A$, $x \in X$. Then $c_0
\vee f(x) = 1$ for every $x \in X$. As $\mathbf{A}$ is totally
ordered, either $c_0 = 1$ or $f(x) = 1$ for every $x \in X$. Thus
(s2$_\ell''$) holds.

Finally, let us show condition (s3). Assume $f * f \leq c$ for some
$f \in B$ and $c \in B \cap C$. Then $f(x)*f(x) \leq c_0$ for every $x
\in X$. Moreover, $f(x) * f(y) \leq c_0$ for every $x,y\in X$, since
$$f(x) * f(y) \leq (f(x) \vee f(y))^2 = f(x)^2 \vee f(y)^2 \leq
c_0.$$ Hence, $f(x) \leq f(y) \implica c_0$ for a fixed $y \in X$ and
every $x \in X$. Thus $(\exists_\vee f)(x) \leq f(y) \implica c_0$
(remember that $\exists_\vee f$ is a constant map). Now $f(y) \leq
(\exists_\vee f)(x) \implica c_0$ for every $y \in Y$. Then, $(\exists_\vee f)(x) \leq (\exists_\vee f)(x) \implica c_0$ and
$(\exists_\vee f)(x) * (\exists_\vee f)(x) \leq c_0$. This concludes
the proof that $\mathbf{B} \cap \mathbf{C}$ is an $m$-relatively complete subalgebra
of $\mathbf{B}$.

By the previous theorem, $\langle \mathbf{B}, \exists_\vee, \forall_\wedge\rangle$ is a monadic BL-algebra. Monadic BL-algebras of this form are called {\em functional monadic BL-algebras}.

Observe that if $\mathbf{A}$ is $|X|$-complete, then $S = A^X$ and $\langle \mathbf{A}^X, \exists_\vee, \forall_\wedge\rangle$ is a functional monadic BL-algebra.
\end{example}

\begin{obs}
Observe that in the previous example the condition that $\mathbf{A}$
is totally ordered was only necessary to prove condition
(s2$_\ell''$) in the definition of $m$-relatively complete
subalgebras. In fact, there exist (non totally ordered)
complete BL-algebras $\mathbf{A}$ for which the subalgebra
$\mathbf{C}$ of constant maps in $\mathbf{A}^X$ is not
$m$-relatively complete. For example, consider $\mathbf{A} = \langle
\mathbb{N}_0, \vee, \wedge, *, \to, 0, 1\rangle$, where
$\mathbb{N}_0$ is the set of nonnegative integers, $\wedge$ is the
least common multiple, $\vee$ is the greatest common divisor, $*$ is
ordinary multiplication and its residuum is given by $$a \implica b =
\begin{cases} 1 & \text{ if } a = b, \\ \frac{b}{gcd(a,b)} & \text{ if } a \ne b.
\end{cases}$$ It is easy to check that $\mathbf{A}$ is a complete
BL-algebra. Now consider the elements $c,f \in A^\mathbb{N}$ given
by $c(n) = 2$, $f(n) = p_n$ (the $n$-th odd prime), for every $n \in
\mathbb{N}$. Then $c \vee f = 1$. However, the only constant map
below $f$ is $0$, and then $c \vee 0 = c \ne 1$. This shows that
condition (s2$''$) does not hold. Observe that we may define an
algebra $\langle \mathbf{A}^\mathbb{N}, \exists_\vee, \forall_\wedge \rangle$ that satisfies all conditions of monadic
BL-algebras but (M4). In particular, this shows that axiom (M4) is
independent of the rest of the axioms.
\end{obs}

\section{Hájek's modal and monadic fuzzy logic}

In his monograph \cite{Hajek98libro} Hájek defined the fuzzy modal logic
\ScBL\ as a modal expansion of his basic logic. Later, in \cite{Hajek10},
he presented an axiomatization for this logic and proved its strong
completeness with respect to its generalized semantics. In this section we will show that monadic BL-algebras are the equivalent algebraic semantics of the logic \ScBL.

The modal logic \ScBL\ is equivalent to the monadic fragment
\mBLpt\ of the fuzzy predicate calculus \BLpt, which
contains only unary predicates and just one object variable $x$ (without object constants). The
propositional variable $p_i$ is associated with the unary predicate
$P_i(x)$ and the modalities $\square$ and $\lozenge$ correspond to
the quantifiers $(\forall x)$ and $(\exists x)$, respectively. For
this reason, and to continue the algebraic tradition of naming {\em
monadic} the algebraic semantics of monadic fragments of several
logics (Boolean, intuitionistic, \L ukasiewicz, etc.), we opted to
call the algebras corresponding to the logic \ScBL\ {\em monadic
BL-algebras}. However, in this section we will work in the language of the modal logic \ScBL\ instead of in the monadic fuzzy language.

ºWe now recall the basic definitions of \ScBL. The axiom schemata  are the ones for the basic logic
\BL\ together with the following modal axiom schemata ($\nu$ stands for any
propositional combination of formulas beginning with $\square$ or
$\lozenge$):

\begin{itemize}
\item[$(\square 1)$] $\square \varphi \implica \varphi$
\item[$(\lozenge 1)$] $\varphi \implica \lozenge \varphi$
\item[$(\square 2)$] $\square (\nu \implica \varphi) \implica (\nu \implica \square \varphi)$
\item[$(\lozenge 2)$] $\square (\varphi \implica \nu) \implica (\lozenge \varphi \implica \nu)$
\item[$(\square 3)$] $\square (\nu \vee \varphi) \implica (\nu \vee \square \varphi)$
\item[$(\lozenge 3)$] $\lozenge ( \varphi * \varphi) \equiv \lozenge \varphi * \lozenge \varphi$
\end{itemize}
where $\varphi \equiv \psi$ stands for $(\varphi \implica \psi) \wedge (\psi \implica \varphi)$.

The inference rules of \ScBL\ are:

\begin{itemize}
\item[(MP)] $\varphi, \varphi \implica \psi \vdash \psi$
\item[(Nec)] $\varphi \vdash \square \varphi$
\end{itemize}

The general semantics of \ScBL\ is given by Kripke models. A {\em
Kripke model} for \ScBL\ is a triple $K = \langle X, e,
\mathbf{A}\rangle$ where $X$ is a nonempty set of worlds,
$\mathbf{A}$ is a BL-chain and $e\colon Prop \times X \to A$ is an
evaluation map, $Prop$ being the set of propositional variables. The
evaluation map extends to any formula:
\begin{itemize}
\item $e(0,x) = 0^\mathbf{A}$, $e(1,x) = 1^\mathbf{A}$.

\item $e(\varphi \wedge \psi, x) = e(\varphi,x) \wedge^\mathbf{A}
e(\psi,x)$, and the same for $\vee$, $\to$ and $*$.

\item $e(\square \varphi, x) = \inf \{e(\varphi,y): y \in X\}$.

\item $e(\lozenge \varphi, x) = \sup \{e(\varphi,y): y \in X\}$.
\end{itemize}
Note that $e(\square \varphi, x)$ and $e(\lozenge \varphi, x)$ may
be undefined. We say that the Kripke model $K$ is {\em safe} if
$e(\square \varphi, x)$ and $e(\lozenge \varphi, x)$ are defined for
every formula $\varphi$.

We write $K \models \varphi$ if $e(\varphi,x) = 1$ for every $x \in
X$. $K$ is a model of a set of formulas $\Gamma$ if $K \models
\varphi$ for every $\varphi \in \Gamma$.

\begin{thm}[Hájek {\cite[Theorem 2]{Hajek10}}]
The modal logic \ScBL\ is strongly complete with respect to its
general semantics, i.e. the following are equivalent for every set
of formulas $\Gamma \cup \{\varphi\}$:
\begin{enumerate}[$(1)$]
\item $\Gamma \vdash \varphi$

\item $K \models \varphi$ for every safe model $K$ of $\Gamma$.
\end{enumerate}
\end{thm}

\begin{obs}
Consider a safe Kripke model $K = \langle X, e, \mathbf{A}\rangle$. Note that we can turn the map $e\colon Prop
\times X \to A$ into a map $\overline{e}\colon Prop \to A^X$ given by the relation $\overline{e}(p)(x) = e(p,x)$. Since $K$ is safe, $\overline{e}$ extends to formulas in
the following way:
\begin{itemize}
\item $\overline{e}(0) = 0^{\mathbf{A}^X}$, $\overline{e}(1) = 1^{\mathbf{A}^X}$,

\item $\overline{e}(\varphi \wedge \psi) = \overline{e}(\varphi) \wedge^{\mathbf{A}^X}
\overline{e}(\psi)$, and the same for $\vee$, $\to$ and $*$,

\item $\overline{e}(\square \varphi) = \forall_\wedge
\overline{e}(\varphi)$,

\item $\overline{e}(\lozenge \varphi) = \exists_\vee
\overline{e}(\varphi)$.
\end{itemize}
Thus, it is clear that $\{\overline{e}(\varphi): \varphi \text{ formula}\} \subseteq A^X$ is the universe of a monadic functional BL-algebra (see Example \ref{EJ: functional MBL}).
\end{obs}

Following the notation of the last remark, we can rewrite Hájek's completeness theorem as follows.

\begin{thm} \label{TEO: completitud 2}
The following are equivalent for every set of formulas $\Gamma \cup
\{\varphi\}$:
\begin{enumerate}[$(1)$]
\item $\Gamma \vdash \varphi$

\item $\overline{e}(\varphi) = 1$ for every $\overline{e}\colon Prop \to B$, where $\langle \mathbf{B}, \exists_\vee, \forall_\wedge\rangle$ is any functional monadic BL-algebra and $\overline{e}(\gamma) = 1$ for every $\gamma \in \Gamma$.
\end{enumerate}
\end{thm}

\begin{thm}
The variety $\mathbb{MBL}$ of monadic BL-algebras is the equivalent
algebraic semantics for the logic \ScBL\ (and \mBLpt).
\end{thm}

\begin{proof}
Following \cite{BlPi89}, it is enough to show the next two conditions for every set of
formulas $\Gamma \cup \{\varphi, \psi\}$:
\begin{enumerate}[({A}LG1)]
\item $\Gamma \vdash \varphi$ if and only if $\{\gamma \approx 1:
\gamma \in \Gamma\} \models_\mathbb{MBL} \varphi \approx 1$.

\item $\varphi \approx \psi \leftmodels \models_\mathbb{MBL} (\varphi \implica \psi) \wedge (\psi \implica \phi) \approx
1$.
\end{enumerate}
Condition (ALG2) is trivially verified. We show condition
(ALG1).

For the forward implication, note that if $\Gamma \vdash \varphi$,
there exists a proof of $\varphi$ from $\Gamma$ and the axioms of
\ScBL\ by successive application of the inference rules (MP) and
(Nec). Thus, it is enough to show that the equation $\varphi
\approx 1$ is valid in $\mathbb{MBL}$ for every axiom $\varphi$ of
\ScBL\ and that the inference rules preserve validity. The former
statement follows from the definition of monadic BL-algebras and Lemma \ref{properties MBL}. The preservation of
(MP) is trivial and the preservation of (Nec) follows from (M\ref{M10}).

For the converse implication, simply observe that, since
$\langle \mathbf{B}, \exists_\vee, \forall_\wedge\rangle \in \mathbb{MBL}$,
condition $(2)$ of Theorem \ref{TEO: completitud 2} holds.
\end{proof}

Thus, from the general theory of algebraic logic, we get the next
corollary.

\begin{cor} \label{COR: corresondencia extensiones y subvariedades}
There is a one-one correspondence between axiomatic extension of
\ScBL\ (or \mBLpt) and subvarieties of $\mathbb{MBL}$.
\end{cor}

From the same theorem we can also derive an important algebraic
result for the variety $\mathbb{MBL}$.

\begin{cor}
The variety $\mathbb{MBL}$ is generated (as a variety) by the functional monadic BL-algebras.
\end{cor}

As a consequence of the algebraization of \ScBL\ by monadic BL-algebras, we may give a simplified set of axioms for this calculus. We define a calculus \ScpBL\ whose axiom schemata are all the ones for basic logic \BL\ together with the following axiom schemata:

\begin{enumerate}[({A}1)]
\item $\square \varphi \implica \varphi$
\item $\square (\varphi \implica \square \psi) \equiv (\lozenge \varphi \implica \square \psi)$
\item $\square (\square \varphi \implica \psi) \equiv (\square \varphi \implica \square \psi)$
\item $\square (\lozenge \varphi \vee \psi) \equiv (\lozenge \varphi \vee \square \psi)$
\item $\lozenge ( \varphi * \varphi) \equiv \lozenge \varphi * \lozenge \varphi$
\end{enumerate}
and the same rules of inference: modus ponens and necessitation. It is easy to prove in the standard way (by means of a Lindenbaum-Tarski algebra) that this calculus is sound and complete with respect to a semantics based on monadic BL-algebras. Indeed, the only non-immediate results are the content of the following lemma.

\begin{lem}
In \ScpBL:
\begin{enumerate}[$(1)$]
\item $\varphi \equiv \psi \vdash \square \varphi \equiv \square \psi$.
\item $\varphi \equiv \psi \vdash \lozenge \varphi \equiv \lozenge \psi$.
\end{enumerate}
\end{lem}

\begin{proof}
To prove $(1)$, note that, by (A1), $\vdash \square \varphi \implica \varphi$, hence, by transitivity of implication, $\varphi \implica \psi \vdash \square \varphi \implica \psi$. Using necessitation, $\varphi \implica \psi \vdash \square(\square \varphi \implica \psi)$, and using (A3), we get $\varphi \implica \psi \vdash \square \varphi \implica \square \psi$. By symmetry, we get $(1)$.

To prove $(2)$, first note that $\vdash \lozenge \varphi \equiv \square \lozenge \varphi$. Indeed, using basic logic theorems, (1), (A4) and (A1), the following equivalences are valid: $\square \lozenge \varphi \equiv \square( \lozenge \varphi \vee \lozenge \varphi) \equiv \lozenge \varphi \vee \square \lozenge \varphi \equiv \lozenge \varphi$.

Note also that $\vdash \psi \implica \lozenge \psi$. Indeed, since $\vdash \lozenge \psi \equiv \square \lozenge \psi$, we get that $\vdash \lozenge \psi \implica \square \lozenge \psi$. Using (A2), $\vdash \square(\psi \implica \square \lozenge \psi)$, and again using that $\vdash \lozenge \psi \equiv \square \lozenge \psi$ and $(1)$, we get that $\vdash \square(\psi \implica \lozenge \psi)$. Finally, using (A1), we obtain that $\vdash \psi \implica \lozenge \psi$.

Now, by transitivity, $\varphi \implica \psi \vdash \varphi \implica \lozenge \psi$. By necessitation, $\varphi \implica \psi \vdash \square(\varphi \implica \lozenge \psi)$. Now, using the equivalences $\square(\varphi \implica \lozenge \psi) \equiv \square (\varphi \implica \square \lozenge \psi) \equiv \lozenge \varphi \implica \square \lozenge \psi \equiv \lozenge \varphi \implica \lozenge \psi$, we get that $\varphi \implica \psi \vdash \lozenge \varphi \implica \lozenge \psi$. By symmetry, we conclude $(2)$.
\end{proof}

\section{Main subvarieties}

In this section we focus our study on three main subvarieties of
$\mathbb{MBL}$: monadic MV-algebras, monadic Gödel algebras and
monadic product algebras. These correspond naturally to the monadic
expansions of the most important extensions of Hájek's basic logic:
\L ukasiewicz logic, Gödel logic and product logic, respectively. In each of these subvarieties we will also give explicit descriptions of their totally ordered structures.

\subsection{Monadic MV-algebras}

MV-algebras are the equivalent algebraic semantics of the infinite-valued \L
ukasiewicz logic (see \cite{CDM00}). It is widely known that they
coincide with involutive BL-algebras. In other words, the variety of
MV-algebras is term-equivalent to the subvariety of $\mathbb{BL}$
determined by the equation $\neg \neg x \approx x$.

In his Ph.\ D. thesis, Rutledge defined and studied monadic
MV-algebras as a way to study the infinite-many-valued \L ukasiewicz
predicate calculus. More recently, in \cite{CiDV14}, we studied in depth
the lattice of subvarieties of these algebras.

In this section we will show that the variety defined by Rutledge
is term-equivalent to the subvariety of $\mathbb{MBL}$ determined by
the equation $\neg \neg x \approx x$.

Let us recall the original definition of monadic MV-algebra.

\begin{dfn}
An algebra $\mathbf{A} = \langle A, \oplus, \neg, \exists,  0
\rangle$ of type $(2,1,1,0)$ is called a \emph{monadic MV-algebra}
(an MMV-algebra for short) if $\langle A, \oplus, \neg, 0\rangle$ is
an MV-algebra and $\exists$ satisfies the following identities:
\begin{multicols}{2}
\begin{enumerate}[(MV1)]
\item\label{MMV1} $x \implica \exists x \approx 1$.
\item\label{MMV2} $ \exists (x\vee y)\approx \exists  x\vee \exists  y$.
\item \label{MMV3}$\exists \neg \exists  x \approx \neg \exists  x$.
\item \label{MMV4} $\exists (\exists x \oplus\exists y)\approx \exists x \oplus \exists y$.
\item\label{MMV5} $\exists (x* x)\approx \exists x*  \exists x$.
\item\label{MMV6} $\exists ( x \oplus x)\approx \exists x\oplus \exists x$.
\end{enumerate}
\end{multicols}
\end{dfn}

In an MMV-algebra $\mathbf{A}$, we define $\forall\colon A \rightarrow A$
by $\forall a =\neg\exists\neg a$, for every $a\in A$. Clearly,
$\exists a = \neg\forall\neg a$. In the following lemma we collect
some properties of MMV-algebras (see \cite{CiDV14}). Recall that on each
(monadic) MV-algebra $\mathbf{A}$ we can define the operations $*$,
$\implica$, $\wedge$ and $\vee$ as follows: $x*y: =\neg( \neg x
\oplus \neg y)$,  $x \implica y := \neg x \oplus y$, $x \wedge y :=
x * (x \implica y)$ and $x\vee y:= \neg (\neg x \oplus y)\oplus y$.

\begin{lem} \label{properties of MMV}
Let $\mathbf{A}=\langle A, \oplus, \neg, \exists,  0 \rangle$ be an
MMV-algebra. For every $a, b\in A$, the following properties hold:
\begin{enumerate}[$(1)$]
\item $\forall a \implica a = 1$.
\item $\forall \neg\forall a = \neg\forall a$.
\item $\forall(\exists a \vee b)= \exists a\vee \forall b$.
\item $\forall (a \implica b) \leq \forall a \implica \forall b$ or, equivalently, $\forall(\neg a \oplus b)\leq \neg\forall a \oplus \forall b$.
\item $\forall(a \wedge b) = \forall a \wedge \forall b$.
\item $\forall(\forall a \oplus \forall b)= \forall a \oplus \forall
b$.
\item $\forall (a * a) = \forall a * \forall a$.
\end{enumerate}
\end{lem}

We collect here two properties of MV-algebras that will be useful
later.

\begin{lem}\label{properties MV}
The following properties hold true in any $MV$-algebra
$\mathbf{A}=\langle A, \oplus, \neg, 0\rangle$, where $a, b$ denote
arbitrary elements of $A$:
\begin{enumerate}[$(1)$]
\item $a \oplus b = (a \implica (a * b)) \implica b$.
\item $(a\implica b)^2 \implica (2 a \implica 2 b) =1$.
\end{enumerate}
\end{lem}

We now turn to derive some useful properties of involutive monadic
BL-algebras. Some of the proofs are inspired by syntactic proofs given by Hájek in \cite{Hajek98libro}. On each BL-algebra $\mathbf{A}$ we define the
operations $\neg$ and $\oplus$ as follows: $\neg x: = x \implica 0$
and $x \oplus y := \neg x \implica y$.

\begin{lem}\label{properties MBL-MV}
The following properties hold in any MBL-algebra $\mathbf{A}$ for every $a \in A$: 
\begin{enumerate}[$(1)$]
\item\label{p101} $\forall \neg \forall a = \neg \forall a$.
\item\label{p102} $\neg \exists a = \forall \neg a$.
\end{enumerate}
Moreover, if $\neg \neg a = a$ for every $a \in A$, the following also hold for arbitrary $a, b$ in $A$:
\begin{enumerate}[$(1)$] \setcounter{enumi}{2}
\item\label{103} $\neg\forall a = \exists\neg a$.
\item\label{104} $a \oplus b = (\forall(a*b)\oplus \neg a)\implica
b$.
\item\label{105} $\forall(a*\exists b)= \forall a * \exists b$.
\item\label{106} $\exists(\exists a \implica b) = \exists a \implica \exists
b$.
\end{enumerate}
\end{lem}
\begin{proof}
\begin{enumerate}[(1)]
\item By (M\ref{M3}), $\forall \neg \forall a = \forall(\forall a\implica 0)= \forall a \implica \forall 0= \forall a \implica 0= \neg \forall
a$.
\item By (M\ref{M9}), $\neg \exists a= \exists a \implica 0 = \exists a \implica \exists 0= \forall(a \implica \exists 0)= \forall \neg
a$.
\item From (2), $\neg \forall a=\neg\forall\neg\neg a = \neg\neg\exists\neg a = \exists\neg
a$.
\item Since $\neg a \leq \forall(a*b)\oplus \neg a$, we have that $\neg a \implica b \geq (\forall(a*b)\oplus \neg a)\implica
b$. On the other hand, using that $\forall(a*b)\oplus\neg a
\leq (a*b) \oplus\neg a$, we get $(\forall(a*b)\oplus\neg a)\implica
b \geq ((a*b) \oplus\neg a)\implica b$ and   $(\neg a \implica
b)\implica ((\forall(a*b)\oplus \neg a) \implica b) \geq (\neg a
\implica b)\implica (((a*b)\oplus \neg a) \implica b)= 1$ from Lemma
\ref{properties MV} (1).
\item Since $\forall a * \exists b \leq a * \exists b$, then $\forall(\forall a * \exists b) \leq \forall(a * \exists
b)$. Thus  $\forall a * \exists b \leq \forall(a * \exists b)$. On
the other hand, we intend to see that $\forall(a * \exists b)
\implica (\forall a * \exists b) = 1$. First observe that $(\neg
\exists b \implica a)\vee(a\implica \neg\exists b)=1$ then, by
(\ref{104}), $1 = ((\forall(a*\exists b)\oplus \neg\exists
b)\implica a)\vee(a\implica \neg\exists b)$. Since $(a\implica
\neg\exists b)\implica \neg\forall(a*\exists b)= \neg(a*\exists
b)\implica \neg\forall(a*\exists b)=1$, we have that
$((\forall(a*\exists b)\oplus \neg\exists b)\implica a)\vee
\neg\forall(a*\exists b)=1$ and $1 = \forall(((\forall(a*\exists
b)\oplus \neg\exists b)\implica a)\vee \neg\forall(a*\exists b))=
\forall((\forall(a*\exists b)\oplus \neg\exists b)\implica a)\vee
\neg\forall(a*\exists b)$, from (M\ref{M12}).

We claim that $\forall((\forall(a*\exists b)\oplus \neg\exists b)\implica a) \leq \forall(a * \exists b) \implica (\forall a * \exists b)$. From this and the fact that $\neg\forall(a*\exists b) \leq \forall(a * \exists b) \implica (\forall a * \exists b)$ it follows immediately that $\forall(a * \exists b) \implica (\forall a * \exists b) = 1$.

It only remains to show that $\forall((\forall(a*\exists b)\oplus \neg\exists b)\implica a) \leq \forall(a * \exists b) \implica (\forall a * \exists b)$. Indeed, using (M\ref{M8}) and (M\ref{M3}), we have that
$\forall((\forall(a*\exists b)\oplus \neg\exists b)\implica a) =
\forall((\exists b \implica \forall (a*\exists b))\implica a)=
\forall(\forall(\exists b \implica (a*\exists b))\implica a)=
\forall(\exists b\implica (a*\exists b))\implica \forall a= (\exists
b\implica \forall(a*\exists b))\implica \forall a$. Since $\forall a
\leq \exists b \implica (\forall a * \exists b)$, it follows that $(\exists b \implica
\forall(a*\exists b))\implica \forall a \leq (\exists b \implica
\forall(a*\exists b))\implica (\exists b \implica (\forall a
*\exists b))= (\exists b*(\exists b \implica \forall(a*\exists
b)))\implica (\forall a
*\exists b) = (\exists b \wedge \forall(a*\exists b))\implica
(\forall a *\exists b)= \forall(a*\exists b)\implica  (\forall a
*\exists b)$.

\item By (\ref{p102}), (\ref{105}), and (\ref{103}),  $\neg \exists (\exists a \implica b) = \forall\neg (\exists a \implica b)=
\forall ((\neg b \implica (\exists a\implica 0))\implica 0) =
\forall (\neg b * \exists a)= \forall\neg
b*\exists a= \neg((\forall\neg b*\exists a)\implica 0)= \neg(\exists
a\implica (\forall \neg b \implica 0))= \neg(\exists a \implica
\neg\forall \neg b) = \neg(\exists a
\implica \exists b)$, then $\exists(\exists a \implica b) = \exists a \implica \exists b$. \qedhere
\end{enumerate}
\end{proof}

We now prove the main theorem of this subsection.

\begin{thm} \label{cuando una MBL es MMV}
The subvariety of $\mathbb{MBL}$ determined by the equation $\neg
\neg x \approx x$ is term-equivalent to the variety of MMV-algebras.
\end{thm}

\begin{proof}
Consider first an MBL-algebra $\mathbf{A} = \langle A, \vee, \wedge, *, \to, \exists, \forall, 0, 1\rangle$ that satisfies the equation $\neg \neg x \approx
x$. We define $\mathbf{A}' = \langle A, \oplus, \neg, \exists,
0\rangle$ where $\neg x := x \implica 0$ and $x \oplus y := \neg x \to
y$. We claim that $\mathbf{A}'$ is an MMV-algebra. Indeed,
conditions (MV1), (MV2), and (MV5) are precisely (M\ref{M7}),
(M\ref{M20}), and (M\ref{M5}), respectively. Condition (MV3) follows from $\exists \neg
\exists a = \exists(\exists a \implica 0)= \exists (\exists a
\implica \exists 0)= \exists a \implica \exists 0 =
 \exists a \implica 0 = \neg \exists a$. From (MV\ref{MMV3}) and (M\ref{M29})  we have that 
$\exists (\exists a \oplus \exists b)= \exists (\neg \exists a
\implica \exists b) = \exists (\exists \neg \exists a \implica
\exists b)= \exists \neg \exists a \implica \exists b= \neg \exists
a \implica \exists b= \exists a \oplus \exists b$; hence, (MV4)
holds. Finally, we prove (MV6).  On the one hand, since $a \oplus a \leq \exists a \oplus \exists a$, we get $\exists(a \oplus a) \leq \exists(\exists a \oplus \exists a) = \exists a \oplus \exists a$.  On the other hand, from Lemma \ref{properties MV} (2), $(\exists a \implica a)^2 \leq (\exists a\oplus \exists a)\implica (a\oplus
 a)$. Then $\exists((\exists a \implica a)^2) \leq \exists((\exists a\oplus \exists a)\implica (a\oplus
 a))$. By (M\ref{M5}), we have that $(\exists(\exists a \implica a))^2 \leq \exists((\exists a\oplus \exists a)\implica (a\oplus
 a))$. From Lemma \ref{properties MBL-MV} (\ref{106}), we conclude that $1 = \exists((\exists a\oplus \exists a)\implica (a\oplus
 a)) = \exists(\exists(\exists a\oplus \exists a)\implica (a\oplus
 a))= (\exists a\oplus \exists a)\implica \exists(a\oplus
 a)$.

Conversely, let $\mathbf{B} = \langle B, \oplus, \neg, \exists,
0\rangle$ be an MMV-algebra and define $\mathbf{B}' = \langle B, \vee, \wedge, *, \to, \exists, \forall, 0, 1\rangle$, where $1 :=
\neg 0$, $x \implica y := \neg x \oplus y$, $\forall x := \neg \exists
\neg x$, $x * y := \neg(\neg x \oplus \neg y)$, $x \vee y :=
\neg(\neg x \oplus y) \oplus y$, $x \wedge y := x * (x \implica y)$. We
claim that $\mathbf{B}'$ is an MBL-algebra that satisfies the
equation $\neg \neg x \approx x$. Indeed, conditions (M\ref{M1}),
(M\ref{M4}) and (M\ref{M5}) are found in Lemma \ref{properties of
MMV} (1), Lemma \ref{properties of MMV} (3), and (MV\ref{MMV5}),
respectively. To prove condition (M\ref{M2}), note that, since $a \leq \exists a$, $\exists a \implica \forall b \leq a \implica \forall b$. Thus $\exists a \implica \forall b = \forall(\exists a \implica \forall b) \leq \forall(a \implica \forall b)$. On the
other hand, $\forall(a \implica \forall b) = \forall(\neg \forall b \implica \neg a) \leq \forall \neg \forall b \implica \forall \neg a = \neg \forall b \implica \neg \exists a = \exists a \implica \forall b$. It only remains to prove condition (M3). By Lemma \ref{properties of MMV}, $\forall(\forall a \implica b) \leq \forall a \implica
\forall b$. In addition, since $\forall a \implica \forall b \leq \forall a \implica b$, it follows that $\forall a \implica \forall b = \forall(\forall a \implica \forall b) \leq \forall(\forall a \implica b)$.
\end{proof}

It would be interesting to characterize all monadic BL-chains and we
will do this in the last section of this article. The problem for
monadic MV-algebras is already solved (see \cite{DNGr04}), but we give here
an elementary proof.

\begin{thm} \label{TEO: en las MMV cadenas los cuantifs son triviales}
Let $\mathbf{A}$ be a totally ordered MMV-algebra. Then $\exists a =
a$ for every $a \in A$, that is, the quantifier on $\mathbf{A}$ is
the identity.
\end{thm}

\begin{proof}
We show the result by way of contradiction. Let $\mathbf{A}$ be a
totally ordered MMV-algebra and assume there exists $a \in A$ such
that $\exists a \ne a$. Consider $b = \exists a \implica a$ and note that
$b \ne 1$ and $\exists b = \exists(\exists a \implica a) = \exists a \to
\exists a = 1$.

Suppose $b^2 \leq \forall b$. Then $1 = (\exists b)^2 = \exists b^2
\leq \exists \forall b = \forall b \leq b$, so $b = 1$, a
contradiction. Since $\mathbf{A}$ is a chain, $\forall b < b^2$.
Then $\forall b \leq \forall b^2 = (\forall b)^2$. Thus $\forall b$
is an idempotent element of $\mathbf{A}$ forcing $\forall b$ to be
$0$ or $1$. The case $\forall b = 1$ implies $b = 1$, a
contradiction. Now assume $\forall b = 0$, that is, $0 = \forall
(\exists a \implica a) = \exists a \implica \forall a$. Since $\forall a \leq
\exists a \implica \forall a$, it follows that $\forall a = 0$. Moreover,
$0 = \exists a \implica \forall a = \neg \exists a = \forall \neg a$.
Now, there are two possibilities for $a$: either $a \leq \neg a$ or
$\neg a \leq a$. If $a \leq \neg a$, then $2(\neg a) = 1$. So $1 =
\forall 2(\neg a) = 2 \forall (\neg a) = 0$, a contradiction.
Analogously, if $\neg a \leq a$, then $2a = 1$ and $1 = \forall 2a =
2 \forall a = 0$, another contradiction.
\end{proof}

\subsection{Monadic Gödel algebras}

Gödel algebras are prelinear Heyting algebras, that is, they
constitute the variety generated by totally ordered Heyting
algebras. Concretely, Gödel algebras are the subvariety of Heyting
algebras determined by the prelinearity equation $(x \implica y) \vee (y
\implica x) \approx 1$. This variety is generated by the Gödel t-norm $[0,1]_G$ and may be also identified with the
subvariety of BL-algebras that satisfy the equation $x^2 \approx x$ (idempotence), since this equation is equivalent to $x * y \approx x \wedge y$. In this subsection we are interested in the subvariety of monadic BL-algebras determined by the equation $x^2 \approx x$. We call these algebras {\em monadic Gödel
algebras}.

Monadic Heyting algebras were introduced by Monteiro and Varsavksy in \cite{MoVa57} and later studied in depth by Bezhanishvili in \cite{Bezhanishvili98}. As in the previous section we will study now the connection between monadic Gödel algebras and the variety of monadic Heyting algebras. More
precisely, we will show that monadic Gödel algebras coincide with
monadic prelinear Heyting algebras that satisfy the equation
\begin{equation} \label{EQ: semidistributividad del para-todo respecto del supremo}
\forall(\exists x \vee y) \approx \exists x \vee \forall y.
\end{equation}

We recall the definition of monadic Heyting algebras from
\cite{Bezhanishvili98}.

\begin{dfn}
An algebra $\langle A, \vee, \wedge, \implica, \exists, \forall, 0, 1 \rangle$, shortened as $\langle \mathbf{A},\exists, \forall\rangle$, is a
\emph{monadic Heyting algebra} if $\langle A, \vee, \wedge, \implica, 0,1 \rangle$ is a Heyting algebra and $\forall, \exists$
are unary operators on $A$ satisfying the following conditions for
all $a,b\in A$.
\begin{enumerate}[(H1)]
\item $\forall a\leq a$, $a\leq \exists a$.
\item $\forall (a\wedge b)= \forall a\wedge \forall b$, $\exists (a\vee b)= \exists a\vee \exists b$.
\item $\forall 1=1$, $\exists 0=0$.
\item $\forall \exists a=\exists a$, $\exists \forall a=\forall a$.
\item $\exists (\exists a\wedge b)=\exists a\wedge \exists b$.
\end{enumerate}
\end{dfn}

The following properties will be useful and their proofs may be
found in \cite{Bezhanishvili98}.

\begin{lem}\label{properties MLH}
The following properties hold in any monadic Heyting algebra
$\mathbf{A}$, where $a, b$ denote arbitrary elements of $A$.
\begin{enumerate}[$(1)$]
\item If $a\leq b$ then, $\forall a\leq \forall b$ and $\exists a \leq \exists
b$.
\item $\forall(a\implica b)\leq \forall a \implica \forall b$.
\item $\forall\forall a = \forall a$ and $\exists\exists a = \exists
a$.
\item $\forall(a\implica \forall b)= \exists a \implica \forall b$.
\item $\forall(a\implica \exists b)= \exists a \implica \exists b$.
\end{enumerate}
\end{lem}

\begin{thm}
Monadic Gödel algebras coincide with monadic prelinear Heyting
algebras that satisfy equation \eqref{EQ: semidistributividad del
para-todo respecto del supremo}.
\end{thm}

\begin{proof}
If $\mathbf{A}$ is a monadic Gödel algebra, from the definition of
monadic BL-algebra and some of the properties of \lemref{properties
MBL} it is immediate that $\mathbf{A}$ is also a monadic prelinear
Heyting algebra and satisfies equation \eqref{EQ:
semidistributividad del para-todo respecto del supremo}.
Conversely, let $\mathbf{A}$ be a monadic prelinear Heyting algebra
that satisfies equation \eqref{EQ: semidistributividad del para-todo
respecto del supremo}. Clearly $\mathbf{A}$ satisfies conditions
(M1), (M4) and (M5) in the definition of monadic BL-algebra.
Condition (M2) follows from Lemma \ref{properties MLH} (4). It
remains to show condition (M3). From Lemma \ref{properties MLH},
$\forall(\forall a\implica b)\leq \forall\forall a \implica \forall
b = \forall a \implica \forall b$. In addition, from (H1) we know
that $\forall b \leq b$. Then, using the properties from Lemma
\ref{properties MLH}, we have that $\forall(\forall a \implica
\forall b)\leq \forall(\forall a\implica b)$ and then $\exists
\forall a \implica \forall b = \forall a \implica \forall b \leq
\forall(\forall a\implica b)$.
\end{proof}


\begin{obs}
Note that monadic prelinear Heyting algebras may not satisfy
equation \eqref{EQ: semidistributividad del para-todo respecto del
supremo}. A counterexample is given by the monadic Heyting algebra
depicted in the Hasse diagram below with the monadic operators
defined as in the table.

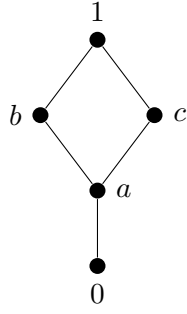
\begin{center}
\begin{minipage}{5cm}
\begin{tikzpicture}[scale=0.5]
\tikzstyle{nodo} = [circle,minimum width = 6pt,fill,inner sep = 0pt]
\node[nodo,label=above:$1$] (v1) at (0,3) {};
\node[nodo,label=left:$b$] (v2) at (-1.5,1) {};
\node[nodo,label=right:$c$] (v3) at (1.5,1) {};
\node[nodo,label=right:$a$] (v4) at (0,-1) {};
\node[nodo,label=below:$0$] (v5) at (0,-3) {};
\draw  (v1) edge (v2);
\draw  (v1) edge (v3);
\draw  (v3) edge (v4);
\draw  (v4) edge (v2);
\draw  (v4) edge (v5);
\end{tikzpicture}
\end{minipage}
\begin{minipage}{5cm}
\begin{tabular}{c|c|c|c|c|c}
$x$ & 0 & $a$ & $b$ & $c$ & $1$ \\ \hline
$\exists x$ & 0 & $c$& 1 & $c$ & 1 \\ \hline
$\forall x $& 0 & $0$ & 0 & $c$ & 1
\end{tabular}
\end{minipage}
\end{center}

Indeed, note that $\forall (b \vee \exists c)=  \forall (b \vee c) = \forall 1 = 1$ whereas $\forall b \vee \exists c = 0 \vee c = c$.
\end{obs}

Finally, we would like to make explicit all the possible monadic
structures that we may define on a given totally ordered Gödel
algebra. Using that $a * b = a \wedge b$ holds true in any Gödel
algebra, the conditions stated in Theorem \ref{TEO: caracterizacion
de la imagen del cuantificador} reduce only to relative
completeness. More precisely, we have the following result.

\begin{thm} \label{TEO: cadenas Heyting}
Given a totally ordered Gödel algebra $\mathbf{A}$ and a relatively
complete subalgebra $\mathbf{C} \leq \mathbf{A}$, if we define on
$A$ the operations $$\forall a := \max \{c \in C: c \leq a\}, \qquad
\exists a := \min \{c \in C: c \geq a\},$$ then $\langle \mathbf{A}, \exists, \forall\rangle$ is a monadic Gödel algebra such that
$\forall A = \exists A = C$.
\end{thm}

In addition, since in any totally ordered Gödel algebra $\mathbf{A}$
we have that $$x \implica y = \begin{cases} 1 & \text{ if } x \leq y, \\
y & \text{ otherwise,} \end{cases}$$ any finite subset of $A$ that
contains $0$ and $1$ is a relatively complete subalgebra of
$\mathbf{A}$ and, hence, defines a structure of monadic Gödel
algebra on $\mathbf{A}$.

\subsection{Monadic product algebras}

In this subsection we introduce the subvariety of $\mathbb{MBL}$ that consists of those monadic BL-algebras whose underlying
BL-structure is a product algebra. We name these algebras {\em
monadic product algebras}. In particular, we will prove that only two monadic operators may be defined on any totally ordered product algebra, namely, the identity operators and the Monteiro-Baaz operators $\Delta$ and $\nabla$ (see \cite{Monteiro80}). We think that this subvariety deserves a more detailed study, but we will intend to pursue this task in a forthcoming article.

Recall that a \emph{product algebra} is a BL-algebra satisfying the
following identities:
\begin{enumerate}[(P1)]
\item $\neg \neg z\implica ((x*z\implica y*z)\implica (x\implica y))\approx 1$.
\item $x\wedge \neg x\approx 0$.
\end{enumerate}

Note that the identity (P2) implies that every product algebra is pseudocomplemented, whereas the identity (P1) implies that in any chain the non-zero elements form a cancellative hoop. For basic properties of product algebras and cancellative hoops, see \cite{Hajek98libro,BlFe00}.


We define a \emph{monadic product algebra} to be a monadic
BL-algebra that is also a product algebra.

\begin{example}
On any product algebra $\mathbf{A}$, we can define two sets of
monadic operators, that we will henceforth call {\em trivial
operators}:
\begin{itemize}
\item Identity operators: $\exists a = \forall a = a$ for every $a \in A$.
\item Monteiro-Baaz operators: $$\forall a = \Delta a =
\begin{cases} 1 & \text{ if } a=1, \\ 0 & \text{ if } a<1, \end{cases} \quad \text{ and } \quad \exists a = \nabla a = \begin{cases} 0 &  \text{ if } a = 0, \\ 1 & \text{ if } a >
0,
\end{cases}$$ for every $a \in A$.

Conditions (M1)-(M4) in the definition of monadic BL-algebras are
easily checked for these operators. To verify condition (M5),
note that if $a \ne 0$, $a^2 \ne 0$ and so $\exists a = (\exists
a)^2 = \exists a^2 = 1$.
\end{itemize}
\end{example}

We now intend to prove that in any totally ordered product algebra
we can only define the trivial quantifiers.

\begin{lem} \label{LEM: props en prod monad}
Let $\mathbf{A}$ be a non-trivial totally ordered monadic product
algebra.
\begin{enumerate}[$(1)$]
\item $\exists(\exists a \implica a) = 1$ for every $a \in A$.
\item If $a \in A \setminus \{0\}$ and $\forall a = 0$, then $\exists a = 1$.
\item If $A' = \{u \in A: \forall u \neq 0\}$, then $\mathbf{A}' = \langle A', *, \implica, 1\rangle$ is a cancellative hoop.
\end{enumerate}
\end{lem}

\begin{proof}
To prove $(1)$, note first that $\exists (\exists 0\implica 0) = 1$.
Now let $a > 0$. Then $a = \exists a \wedge a = \exists a * (\exists
a \implica a)$. Thus, \[\exists a = \exists ( \exists a*(\exists
a\implica a))= \exists a* \exists(\exists a\implica a).\] Since
$\exists a > 0$, we can cancel out $\exists a$ in the previous
equation and get that $1=\exists(\exists a\implica a)$.

To prove $(2)$, let $a \in A\setminus\{0\}$ such that $\forall a =
0$. Since $\mathbf A$ is totally ordered then $\exists a\implica
a<\exists a$ or $\exists a\leq \exists a\implica a$. Let us suppose
that $\exists a\leq \exists a\implica a$. Then $0<a\leq \exists
a=\forall \exists a\leq \forall (\exists a\implica a)= \exists
a\implica \forall a=\exists a\implica 0=0$, a contradiction. Then
$\exists a\implica a<\exists a$. From $(1)$, we have that $1 =
\exists (\exists a \implica a) \leq \exists \exists a$, so $\exists
a=1$.

Finally we show $(3)$. Since the non-zero elements of $\mathbf{A}$
form a cancellative hoop, it is enough to show that $1 \in A'$ and
that $A'$ is closed under $*$ and $\to$. Indeed, $1 \in A'$ since
$\forall 1=1 > 0$. In adddition, if $u,v\in A'$, then
\[0<\forall u*\forall v=\forall(\forall u*\forall v)\leq \forall (u*v).\]
Thus, $u*v\in A'$. From $0<\forall v\leq \forall (u\implica v)$, we
have that $u\implica v\in A'$.
\end{proof}

From the last lemma, we know that $\mathbf{A}'$ is a cancellative hoop. In
particular, $\mathbf{A}'$ is a Wajsberg hoop and hence, we may apply to this
hoop the MV-closure construction given in
\cite{AbCaDV10}. We recall here
this construction.
%
We define the
MV-closure $\mathbf{MV(A')}$ of the cancellative hoop $\mathbf{A'} =
\langle A', *, \implica ,1\rangle$ as follows:
\[\mathbf{MV(A')}= \langle A'\times \{0,1\}, \oplus_{mv}, \neg_{mv}, 0_{mv}\rangle,\]
where $0_{mv}:=(1,0)$, $\neg_{mv}(a, i):=(a, 1-i)$,
\[(a,i)\oplus_{mv}(b,j):=\begin{cases}
(a\oplus b,1 ) &\text{if } i=j=1,\\
(b\implica a,1 ) &\text{if } i=1 \text{ and } j=0,\\
(a\implica b,1 ) &\text{if } i=0 \text{ and } j=1,\\
(a * b,0 ) &\text{if } i=j=0,
\end{cases}\]
where $a\oplus b= (a\implica(a*b))\implica b$. The other common
MV-operations are defined as follows:
\begin{quote}
\begin{itemize}
\item $1_{mv} := \neg_{mv}0_{mv}=(1, 1)$,
\item $(a,i)  * _{mv} (b,j) := \neg_{mv}(\neg_{mv}(a,i) \oplus_{mv}
\neg_{mv}(b,j))$,
\item $(a, i) \implica_{mv} (b,j) := \neg_{mv}(a,i) \oplus_{mv} (b,j)$,
\item $(a,i) \wedge_{mv} (b,j) := (a,i)  * _{mv} ((a,i)
\implica_{mv} (b,j))$,
\item $(a,i) \vee_{mv} (b,j) = ((a,i) \implica_{mv} (b,j)) \implica_{mv} (b,j)$.
\end{itemize}
\end{quote}
If we identify $a$ with $(a,1)$ for every $a\in A'$, we can consider
$\mathbf A'$ a subalgebra of the hoop-reduct of $\mathbf{MV(A')}$.
If $b=(a, 0)$ for some $a\in A'$, then $b=\neg_{mv}(a,1)$ and we
write $b=\neg_{mv}a$. Then we can consider the universe of $\mathbf{MV(A')}$ as the
disjoint union of $A'$ and $\neg_{mv} A'$. The order relation on
$\mathbf{MV(A')}$ is given by: $(a,i)\leq (b,j)$ if and only if one
of the following conditions holds:
\begin{quote}
\begin{itemize}
\item $a\leq b$ and $i=j=1$,
\item $a\oplus b=1$, $i=0$, and $j=1$,
\item $b\leq a$ and $i=j=0$.
\end{itemize}
\end{quote}
Observe that since $\mathbf{A'}$ is a cancellative hoop, then
$a\oplus b=1$ for every $a, b\in A'$, so in this case, the
elements in $A'$ are all above the ones in $\neg_{mv}A'$. Thus
$\mathbf{MV}(\mathbf{A}')$ is a totally ordered MV-algebra. Let us define the following quantifiers:
\begin{multicols}{2}
$\exists (a, i)=
\begin{cases}
(\exists a, 1) & \text{if } i=1, \\
(\forall a, 0) & \text{if } i=0,
\end{cases}$

$\forall (a, i)=
\begin{cases}
(\forall a, 1) & \text{if } i=1,\\
(\exists a, 0) & \text{if } i=0.
\end{cases}$
\end{multicols}
Then the following result may be easily checked.

\begin{lem}
Let $\mathbf{A}$ be a non-trivial totally ordered monadic product
algebra and let $A' = \{u\in A: \forall u \neq 0\}$. Then $\langle
\mathbf{MV}(\mathbf{A}'), \exists,\forall\rangle$ is an MMV-chain.
\end{lem}

As we saw in the previous subsection, on an MV-chain we can only define the identity quantifiers. Thus the next corollary is
immediate.

\begin{cor} \label{cor:e(A')=f(A')=A'}
Let $\mathbf{A}$ be a non-trivial totally ordered monadic product
algebra, then $\exists A' = \forall A' = A'$.
\end{cor}

We are now ready to prove the main result of this subsection.

\begin{thm} \label{TEO: cadenas producto}
In any totally ordered monadic product algebra the quantifiers are
trivial.
\end{thm}

\begin{proof}
Let $\mathbf{A}$ be a non-trivial totally ordered monadic product
algebra. Let us suppose that there is $u\in A$ such that $\forall
u=0$ and $u>0$. From \lemref{LEM: props en prod monad} $(2)$ we know
that $\exists u=1$. Let $v \in A$ and $v \neq 0,1$. Since $\mathbf{A}$ is a
chain we have that $v\leq u $ or $u\leq v$. If $v \leq u$, then
$\forall v=0$ and, again by \lemref{LEM: props en prod monad} $(2)$,
$\exists v=1$. If $u\leq v$ then $\exists v=1$. If $\forall v\neq 0$
then $v\in A'$ and, from the previous corollary, $\forall
v=v=\exists v=1$, which is a contradiction. So, $\forall v=0$. This
proves that $\exists = \nabla$ and $\forall = \Delta$.

We now deal with the case in which $\forall u\neq 0$ for all $u \in
A\setminus\{0\}$. Then $A' = \{u\in A: \forall u \neq 0\} = A
\setminus \{0\}$ and, from the previous corollary, we have that
$\exists A' = \forall A' = A'$. Therefore, $\exists = \forall = id$.
\end{proof}

\section{Monadic BL-chains}

The objective of this section is to characterize all MBL-chains. Based on the characterization of BL-chains as ordinal sums of totally ordered Wajsberg hoops given by Aglianò and Montagna in \cite{AgMo03}, later simplified by Busaniche in \cite{Busaniche04}, we will present a way of building a monadic BL-chain as an ordinal sum of totally ordered Wajsberg hoops indexed on a monadic Heyting chain. Moreover, we will also show that any monadic BL-chain may be obtained using this construction.

First we need to recall the definition of ordinal sum of a family of
Wajsberg hoops indexed on a totally ordered set $I$. Fix a bounded
totally ordered set $(I,\leq,0,1)$ that will be used as an index
set. Set $C_1 = \{1\}$ and, for each $i \in I \setminus\{1\}$, let
$C_i$ be a set such that $\langle C_i \cup \{1\}, *_i, \implica_i, 1\rangle$ is a totally ordered Wajsberg hoop. Assume also that $C_0$
has a least element $0$.

Let $C_I := \{(i, a): i \in I, a \in C_i\} \subseteq I \times
(\bigcup_{i \in I} C_i)$. We define a total order on $C_I$
(lexicographic order) as follows: \[ (i,a) \leq (j,b) \text{ iff } i
< j \text { or } (i = j \text { and } a \leq b).\] We define on
$C_I$ the following operations:

\vspace{-1cm}
\begin{multicols}{2}
\[(i, a) * (j,b)
= \begin{cases}
(i, a) & \text{ if } i<j, \\
(i, a*_i b) & \text{ if } i=j, \\
(j, b) & \text{ if } i>j,
\end{cases}\]

\[(i, a)\implica (j,b) =
\begin{cases}
(1, 1) & \text{ if } (i, a)\leq (j,b), \\
(i, a\implica_i b) & \text{ if } i=j \text{ and } a>b, \\
(j, b) & \text{ if } i>j,
\end{cases}\]
\end{multicols}

The algebra $\mathbf{C}_I = \langle C_I, \vee, \wedge,
*, \implica, (0,0), (1,1)\rangle$ is a BL-chain. This follows
immediately from the construction given by Busaniche in \cite{Busaniche04}.

In order to define a structure of monadic BL-algebra on
$\mathbf{C}_I$, we will require that $I$ be endowed with a monadic
Heyting structure; we denote by $\mathbf{I}$ this monadic Heyting
chain. We also need to assume that the sets $C_i$, $i \in I$,
satisfy the following conditions:
\begin{quote}
\begin{itemize}
\item if $\forall i < i$, $C_{\forall i}$ has a greatest element $u_{\forall i}$,
\item if $i < \exists i$, $C_{\exists i}$ has a least element $0_{\exists i}$.
\end{itemize}
\end{quote}
Note that if $\exists i = 1$, then $0_{\exists i} = 0_1 = 1$.

Consider the subset $S \subseteq C_I$ given by: $(i,a) \in S$ iff $i
\in \forall I$. We claim that $\mathbf{S}$ is an $m$-relatively complete
subalgebra of $\mathbf{C}_I$. The fact that $\mathbf{S}$ is a subalgebra of
$\mathbf{C}_I$ is immediate from the definition of the operations.
We now show that the conditions for $m$-relative completeness hold
for $\mathbf{S}$:
\begin{enumerate}[$(1)$]
\item Let $(i,a)$ be any element of $C_I$. If $i \in \forall I$,
then $(i,a) \in S$ and the conditions hold trivially. Suppose $i
\not\in \forall I$. Then $i < \exists i$ and it follows that
$(\exists i, 0_{\exists i})$ exists and $(i,a) \leq (\exists i,
0_{\exists i})$. Moreover, if $(i,a) \leq (j,b)$ for some $(j,b) \in
S$, i.e., $j \in \forall I$, then $i \leq j$, so $\exists i \leq j$
and then $(\exists i, 0_{\exists i}) \leq (j,b)$. The dual
condition follows analogously.

\item This condition is trivial for chains.

\item Consider $(i,a) \in C_I$ and $(j,b) \in S$ (i.e. $j \in \forall I$) such that $(i,a)^2
= (i,a^2) \leq (j,b)$. If $i \in \forall I$, then $(i,a) \in S$ and
the condition follows trivially. Suppose $i \not\in \forall I$.
Then $i < \exists i \leq j$, so $(i,a) \leq (\exists i, 0_{\exists
i})$ and $(\exists i, 0_{\exists i})^2 = (\exists i, 0_{\exists i})
\leq (j,b)$.
\end{enumerate}
We denote by $\mathbf{C_I}$ the monadic BL-algebra induced by $S$ on
$\mathbf{C}_I$.

We have thus shown the following theorem.

\begin{thm}
$\mathbf{C_I}$ is a monadic BL-chain and the monadic operators on
$\mathbf{C_I}$ are given by:

\vspace{-1cm}
\begin{multicols}{2}
\[\exists(i,a) = \begin{cases}
(i,a) & \text{ if } i \in \forall I, \\
(\exists i, 0_{\exists i} ) & \text{ if } i \notin \forall I,
\end{cases}\]

\[ \forall(i,a)= \begin{cases}
(i, a) & \text{ if } i \in \forall I, \\
(\forall i, u_{\forall i} ) & \text{ if } i \notin \forall I.
\end{cases}\]
\end{multicols}
\end{thm}

We will now show that {\em any} monadic BL-chain is isomorphic to
$\mathbf{C_I}$ for a suitable monadic Heyting chain $\mathbf{I}$ and
suitable Wajsberg chains $\{\mathbf{C}_i: i \in I\}$.

Fix a monadic BL-chain $\mathbf{A}$ and recall the representation of
BL-chains as ordinal sums of totally ordered Wasjberg hoops given by
Busaniche in \cite{Busaniche04}.

For each $a \in A$, we consider the set $F_a = \{x \in A \setminus
\{1\}: a \implica x = x\}$. We can define the following equivalence
relation on $A$: $a \sim b$ iff $F_a = F_b$. Each equivalence class
$C$ is a convex set and $\mathbf{C'} = \langle C \cup \{1\}, *,
\implica, 1 \rangle$ is a totally ordered Wajsberg hoop. Let $I$ be
the set of equivalence classes ordered by: $C \preceq D$ iff either
$C = D$, or, for all $a \in C$ and for all $b \in D$, $a \leq b$. We
write $C \prec D$ when $C \preceq D$ and $C \neq D$. We also know
that if $C \prec D$, $a \in C$ and $b \in D$, then $b \implica a =
a$ and $a * b = a$. We denote by $C_0$ the equivalence class that
contains the element 0, and by $C_1$ the class that contains the
element 1. Observe that $C_1 = \{1\}$. Then $\mathbf{A}$ is
isomorphic as a BL-algebra to $\mathbf{C}_I$ as defined above.

For each equivalence class $C$ we will show that either $C \subseteq
\forall A$ or $C \cap \forall A = \emptyset$. Since this is
trivially true for the class $C_1$, we assume $C \ne C_1$ in the
sequel.

We now show that $C \cap \forall A$ is an increasing subset of $C$. Suppose there is $c \in C$ such that $\forall c = \exists c = c$ and consider $D = \{a \in C: a \geq c\} \cup \{1\}$. Since $D$ is an bounded increasing subset of $C'$, we can define an MV-structure on $D$. Indeed, $\mathbf{D} = \langle D, \vee, \wedge, *_c, \implica, c, 1\rangle$ where $x *_c y := (x * y) \vee c$. Note that if $a \in D$, $a \ne 1$, then $c \leq a$ and $c = \forall c \leq
\forall a \leq a$, and since $C$ is convex, $\forall a \in D$. It follows that $D$ is closed under $\forall$. We define a unary operation $\exists'$ on $D$ by
$$\exists' x := \forall(x \implica c) \implica c = (\exists x \implica c) \to c.$$

We claim that $\langle \mathbf{D}, \exists', \forall\rangle$ is a
totally ordered monadic MV-algebra. To prove this, it is enough to
check the identities (M1)-(M5).

Since the identities (M1) and (M3) involve only the operations $\to,
\forall, 1$, and these are defined on $\langle \mathbf{D}, \exists', \forall\rangle$ as restrictions of the original operations of
$\mathbf{A}$, these identities hold trivially on $\langle
\mathbf{D}, \exists', \forall\rangle$.

Fix $a,b \in D$. If $\exists a \in D$, then $\exists' a =
(\exists a \implica c) \implica c = \exists a \vee c = \exists a$ and it is
clear that 
\begin{align}
\forall(a \implica \forall b) & = \exists' a \to
\forall b, \\
\forall(\exists'a \vee b) & = \exists' a \vee \forall
b.
\end{align}
If $\exists a \not\in D$, then $\exists'a
= (\exists a \implica c) \implica c = c \implica c = 1$. Then $(3)$ holds
trivially, and we also have that $\forall(a \implica \forall b) = \exists
a \implica \forall b = \forall b = \exists'a \implica \forall b$. This shows
that $\langle \mathbf{D}, \exists', \forall\rangle$ satisfies (M2)
and (M4).

It remains to show the validity of (M5). Take $a \in D$ and note that $\exists' (a *_c a) = \exists'(a^2 \vee c) = (\exists(a^2 \vee c) \implica c) \implica c = (\exists a^2 \implica c) \implica c$ and $(\exists' a) *_c (\exists' a) = ((\exists a \implica c) \implica c)^2 \vee c$. We distinguish the following cases:
\begin{itemize}
\item $\exists a \not\in D$: In this case $\exists a \implica c = c$, so $\exists' a = 1$. Moreover, $\exists a^2 = (\exists a)^2 \not\in D$. Hence, $\exists' a *_c \exists' a = \exists' (a *_c a) = 1$.

\item $\exists a \in D$ and $\exists a^2 \in D$: In this case, $\exists'(a *_c a) = \exists a^2 \vee c = \exists a^2$ and $\exists' a *_c \exists' a = (\exists a \vee c)^2 \vee c = (\exists a)^2 \vee c = \exists a^2$.

\item $\exists a \in D$, but $\exists a^2 \not\in D$: In this case $a \ne 1$, $\exists a \ne 1$, and $\exists a^2 = (\exists a)^2 < c$. Thus $\exists' (a *_c a) = c$ and $\exists' a *_c \exists' a = (\exists a)^2 \vee c = c$.
\end{itemize}
This concludes the verification that equation (M5) holds in $\langle \mathbf{D}, \exists', \forall\rangle$.

We have thus shown that $\langle \mathbf{D}, \exists', \forall\rangle$
is a totally ordered monadic MV-algebra. From Theorem \ref{TEO: en
las MMV cadenas los cuantifs son triviales}, we deduce that $\exists'$, and also $\forall$ and $\exists$, are the identity quantifiers on $D$. Thus $D \subseteq \forall A$. This proves the claim that $C \cap \forall A$ is an increasing subset of $C$.

We now distinguish two possibilities for the totally ordered Wajsberg hoop $\mathbf{C}'$.

\begin{itemize}
\item $\mathbf{C}'$ is a bounded Wajsberg hoop, that is, there exists $0' \in C$ such that $0' \leq a$, for all $a \in C$.

In this case we can endow $\mathbf{C}'$ with a natural MV-structure.
We make a further distinction according to whether $0' \in \forall
A$ or not.

\begin{itemize}
\item Assume $0' \in \forall A$.

In this case, since $C \cap \forall A$ is an increasing subset of $C$, it follows that $C \subseteq \forall A$.

\item Assume $0' \not\in \forall A$.

In this case $\forall 0' < 0' < \exists 0'$. We will show that $C
\cap \forall A = \emptyset$. Indeed, assume there is $a \in C$ such
that $\forall a = \exists a = a$. As $0' \leq a$, we have that $0'
\leq \exists 0' \leq \exists a = a$, so $\exists 0' \in C$. Note
that $\exists 0' \implica 0' \neq 1$ and $\exists 0' \implica 0'=
\exists (0'*0') \implica 0'= (\exists 0'* \exists 0') \implica 0'=
\exists 0' \implica (\exists 0' \implica 0')$, hence $\exists 0'
\implica 0'\in F_{\exists 0'}= F_{0'}$, and then $0'\implica
(\exists 0' \implica 0')= \exists 0'\implica 0'$, which is not
possible since $0'\implica (\exists 0' \implica 0')=1$ and $\exists
0'\implica 0'\neq 1$.

Observe that, since $C \cap \forall A = \emptyset$, for every $a \in
C$, $\forall a = \forall 0'$ and $\exists a = \exists 0'$.
\end{itemize}

\item $\mathbf{C'}$ is an unbounded Wajsberg hoop, and hence a cancellative hoop.

We claim that for any $a \in C$: if $\forall a \not\in C$, then
$\exists a \not\in C$.

By way of contradiction, assume that $\forall a \not\in C$ and
$\exists a \in C$. Note that $\forall a < a < \exists a$ and hence
$\exists a \implica a \in C$. Observe also that, since $a = a \wedge
\exists a = \exists a * (\exists a \implica a)$, we have that $\exists a
= \exists a * \exists(\exists a \implica a)$. If $\exists(\exists a \to
a) \in C$, then, using the cancellative property, we get that
$\exists(\exists a \implica a) = 1$, a contradiction. This shows that
$\exists(\exists a \implica a) \not\in C$ and, consequently,
$\exists(\exists a \implica a) > b$ for every $b \in C$.

Suppose now that $\exists a \implica a \leq \exists a$. Then
$\exists(\exists a \implica a) \leq \exists a$, which is a contradiction.
On the other hand, suppose that $\exists a \leq \exists a \implica a$. In
this case, $\exists a^2 = (\exists a)^2 \leq a$ and, moreover,
$\exists a^2 \leq \forall a$. However $\forall a \not\in C$ and
$\forall a < b$ for every $b \in C$. Thus $\forall a < a^2 \leq
\exists a^2 \leq \forall a$, a contradiction.

Consequently, if there is $a \in C$ such that $\forall a \not\in C$, then $\exists a \not\in C$ and it is clear that $C \cap \forall A = \emptyset$. Otherwise, if, for every $a \in C$, $\forall a \in C$, then, using the fact that $C \cap \forall A$ is an increasing subset of $C$, it follows that $C \subseteq \forall A$.

\end{itemize}

From what we have just proved, it follows that, for any given $a \in
A$, there are two possibilities:
\begin{itemize}
\item if $C_a \subseteq \forall A$, then $\forall a =
\exists a = a$;

\item if $C_a \cap \forall A = \emptyset$, then $C_{\forall a} \prec
C_a \prec C_{\exists a}$; moreover, since $C_{\forall a}, C_{\exists
a}$ must be contained in $\forall A$, then $\forall a$ is the
greatest element of $C_{\forall a}$ and $\exists a$ is the least
element of $C_{\exists a}$.
\end{itemize}

Finally, note that we can define a monadic Heyting structure on $I =
A/\mathord{\sim} = \{C_a: a \in A\}$. Indeed, since $(I, \preceq)$
is a bounded totally ordered set with least element $C_0$ and
greatest element $C_1$, it can be turned into a totally ordered
Heyting structure $\mathbf{I}$:

\medskip

\hspace{2.5cm}
\begin{minipage}{5cm}
$C_a \wedge C_b = C_{a \wedge b}$,

\smallskip

$C_a \vee C_b = C_{a \vee b}$,
\end{minipage}
\begin{minipage}{7cm}
$C_a \implica C_b = \begin{cases} C_1 & \text{ if } C_a \preceq C_b, \\ C_b
& \text{ if } C_b \prec C_a.\end{cases}$
\end{minipage}

Now consider $S = \{C_a: a \in \forall A\}$. It is immediate that $\mathbf{S}$ is a subalgebra of $\mathbf{I}$. Moreover, $\mathbf{S}$ is a
relatively complete subalgebra. Indeed, let $C_a$ be an arbitrary
element of $I$ and $C_b \in S$ such that $C_b \preceq C_a$. There
are two possible situations. If $C_b = C_a$, then $C_a \subseteq
\forall A$, so $\forall a = a$ and $C_b = C_a = C_{\forall a}$. If
$C_b \prec C_a$, then $b < a$, so $b \leq \forall a$ and $C_b
\preceq C_{\forall a}$. This shows that $C_{\forall a}$ is the
greatest element in $S$ below $C_a$. In an analogous manner, it may
be shown that $C_{\exists a}$ is the least element in $S$ above
$C_a$.

Since $\mathbf{S}$ is a relatively complete subalgebra of $\mathbf{I}$, by
Theorem \ref{TEO: cadenas Heyting}, $S$ defines a monadic Heyting
structure on $\mathbf{I}$.

If we consider now the map $\psi\colon A \to C_I$ given by $\psi(a) =
(C_a,a)$, it is clear from what we have shown above that $\psi$ is
an isomorphism of monadic BL-algebras.

We have thus finished the proof of the characterization of all
monadic BL-chains.

\begin{thm}
Any monadic BL-chain is isomorphic to some $\mathbf{C_I}$.
\end{thm}

To close this section, we would like to remark that the subvariety
of $\mathbb{MBL}$ generated by monadic BL-chains may be axiomatized
within $\mathbb{MBL}$ by a single identity: $$\forall (x \vee y) \approx
\forall x \vee \forall y.$$ Indeed, it is easily verified that any
monadic BL-chain satisfies this identity. Conversely, consider a
subdirectly irreducible algebra $\mathbf{A} \in \mathbb{MBL}$ that
satisfies the identity. Given $x,y \in A$ we know that $(x \implica y)
\vee (y \implica x) = 1$. Then $\forall(x \implica y) \vee \forall (y \implica x) =
1$, but, since $\forall A$ is totally ordered, either $\forall(x \to
y) = 1$ or $\forall(y \implica x) = 1$. It follows that $x \leq y $ or $y
\leq x$. This proves that $\mathbf{A}$ is totally ordered.

\subsection*{Conclusions and further work}

In this work we have presented an equivalent algebraic semantics for the monadic fragment of Hájek's fuzzy predicate logic. This class turned out to be a very interesting variety, whose algebras we called {\em monadic BL-algebras}.

We have started a study of some subvarieties of $\mathbb{MBL}$, but we think that a deeper study of many of its subvarieties is in order. For example, it would be interesting to know the lattice of subvarieties of monadic product algebras, Gödel algebras, and the variety generated by chains. In each case, we would like to determine whether they are generated by their finite members.

Another variety worth of research is the one generated by functional monadic BL-algebras defined over continuous t-norms. This is a proper variety because it satisfies the equation $\forall (x * \forall y) \approx \forall x * \forall y$.

\end{document}